\documentclass[12pt,a4paper]{amsart}

\usepackage{tikz}
\usepackage{amssymb}
\usepackage{amsthm}
\usepackage{amsmath}
\usepackage{latexsym}
\usepackage{amsfonts}
\usepackage{color}
\usepackage{stmaryrd}
\usepackage{epic,latexsym,bezier,caption,amsbsy,psfrag,color,enumerate,amsfonts,amsmath,amscd,amssymb}
\usepackage{epsfig}
\usepackage{rotating}
\usepackage{graphics}
\usepackage[latin1]{inputenc}
\usetikzlibrary{calc}

\newtheorem{theorem}{Theorem}[section]
\newtheorem{proposition}{Proposition}[section]
\newtheorem{definition}{Definition}[section]
\newtheorem{lemma}{Lemma}[section]
\newtheorem{corollary}{Corollary}[section]
\newtheorem{rem}{Remark}[section]
\newtheorem{example}{Example}[section]

\begin{document}

\keywords{Graph automaton group, Schreier graph, Growth, Fractal group, Self-similar representation, Diameter, Graph automorphism.}

\title{Graph automaton groups}

\author{Matteo Cavaleri}
\address{Matteo Cavaleri, Universit\`{a} degli Studi Niccol\`{o} Cusano - Via Don Carlo Gnocchi, 3 00166 Roma, Italia}
\email{matteo.cavaleri@unicusano.it}

\author{Daniele D'Angeli}
\address{Daniele D'Angeli, Universit\`{a} degli Studi Niccol\`{o} Cusano - Via Don Carlo Gnocchi, 3 00166 Roma, Italia}
\email{daniele.dangeli@unicusano.it}

\author{Alfredo Donno}
\address{Alfredo Donno, Universit\`{a} degli Studi Niccol\`{o} Cusano - Via Don Carlo Gnocchi, 3 00166 Roma, Italia}
\email{alfredo.donno@unicusano.it}

\author{Emanuele Rodaro}
\address{Emanuele Rodaro, Politecnico di Milano - Piazza Leonardo da Vinci, 32 20133 Milano, Italia}
\email{emanuele.rodaro@polimi.it}

\begin{abstract}
In this paper we define a way to get a bounded invertible automaton starting from a finite graph. It turns out that the corresponding automaton group is regular weakly branch over its commutator subgroup, contains a free semigroup on two elements and is amenable of exponential growth. We also highlight a connection between our construction and the right-angled Artin groups. We then study the Schreier graphs associated with the self-similar action of these automaton groups on the regular rooted tree. We explicitly determine their diameter and their automorphism group in the case where the initial graph is a path. Moreover, we show that the case of cycles gives rise to Schreier graphs whose automorphism group is isomorphic to the dihedral group. It is remarkable that our construction recovers some classical examples of automaton groups like the Adding machine and the Tangled odometer.
\end{abstract}

\maketitle

\begin{center}
{\footnotesize{\bf Mathematics Subject Classification (2010)}: 20F65, 20F05, 20E08, 05C10, 05C25.}
\end{center}

\section{Introduction}
Algebraic structures can be usually described by means of their combinatorial nature. A typical example is the relation between groups and graphs. In Geometric Group Theory, for instance, many algebraic properties of a group can be detected by investigating the combinatorial properties of the corresponding Cayley graphs. Another typical example of this interaction is achieved by the automaton group theory. Automata (or Mealy machines or transducers) are directed graphs whose transitions describe the action of the states on a finite alphabet. If, for instance, the automaton is complete and invertible, then its states generate a group that has, in many cases, very interesting properties.\\
\indent As an example, in 1980 R. I. Grigorchuk described in \cite{gri80} the first group of intermediate (i.e., faster than polynomial and slower than exponential) growth, that later appeared to be generated by a finite automaton. This group has a number of other interesting properties: for example, it is infinite and finitely generated, but each of its elements has finite order (Burnside group). It was also the first example of an amenable but not elementary amenable group. Over the last decades, a new exciting direction of research focusing on finitely generated automaton groups acting by automorphisms on rooted trees has been developed. It turned out that it has deep connections with the theory of profinite groups and with complex dynamics. The interested reader can refer for more details to the following list of works (and bibliography therein) \cite{fractal_gr_sets, bhn:aut_til, dynamicssubgroup, nekrashevyc}. Recently a special interest has been pointed out for decision problems for automaton groups and semigroups (see \cite{freeness, gillibert, gill, con} and references therein).\\
\indent Beside Cayley graphs the best way to encode the action of automaton groups is given by the structure of the corresponding Schreier graphs. These graphs better represent the length-preserving action of an infinite automaton group on a finite alphabet and constitute an infinite sequence of finite graphs converging to infinite graphs in the Gromov-Hausdorff topology. Finite Schreier graphs can be regarded as orbital graphs with respect to the action of the automaton group on each level of a regular rooted tree, whereas their limits are orbital graphs of the action of the group on the boundary of the tree. Finite Schreier graphs have been studied from a combinatorial point of view in several contexts (e.g., the investigation of models coming from Statistical Mechanics \cite{ising, dimeri}). Classification of infinite Schreier graphs have been studied in several papers (see \cite{intermediate, bound, mio, basilica, infinito, auto} for further discussions about this topic).

In this paper, we introduce a family of automaton groups arising from finite graphs, that we call {\it graph automaton groups} (Definition \ref{maindefi}). More precisely, with any finite graph $G=(V,E)$, we associate an invertible automaton $\mathcal{A}_G$ whose state set is identified with the edge set $E$ and that acts on words over an alphabet identified with the vertex set $V$. Basically, any edge has a nontrivial action only on words starting with a letter identified with one of its endpoints. In this way, we can define an automaton that has the remarkable property of being bounded. The class of bounded automaton groups is very popular and it contains the most famous examples of automaton groups (e.g., the Grigorchuk group and the Basilica group \cite{grizuk}). It is a remarkable fact that all groups generated by bounded automata are amenable, and so all our graph automaton groups have this property. It is interesting that, with our construction, we recover some classical examples of automaton groups (e.g., the Adding machine and the Tangled odometer group, see Example \ref{esempibase}). Among other properties, we prove that our groups are fractal and weakly regular branch over their commutator subgroup (Theorem \ref{teo_principale}), that they all contain (except some trivial cases) a free semigroup (Theorem \ref{teo_esp}) and so they have exponential growth (Corollary \ref{corogrowthexp}). We also highlight a connection between our groups and the right-angled Artin groups (Proposition \ref{artino}).\\
\indent The second part of the paper is devoted to the investigation of the Schreier graphs associated with the action of graph automaton groups. In Theorem \ref{cicli}, we give a rigidity result for the automorphism group of Schreier graphs when the initial graph is cyclic. Then we present an explicit recursive description of the Schreier graphs of graph automaton groups generated by path graphs, and we are able to determine their diameter (Theorem \ref{diametro}) and their automorphism group (Theorem \ref{autogroupfinale}).

\section{Preliminaries}\label{sec: preliminaries}
We recall in this preliminary section some basic definitions and properties about automaton groups and growth.
\begin{definition}
\indent An \textit{automaton} is a quadruple $\mathcal{A} =
(S,X,\lambda,\mu)$, where:
\begin{enumerate}
\item $S$ is the set of states;
\item $X=\{1,2,\ldots, k\}$ is an alphabet;
\item $\lambda: S\times X \rightarrow S$ is the restriction map;
\item $\mu: S\times X \rightarrow X$ is the output map.
\end{enumerate}
\end{definition}
The automaton $\mathcal{A}$ is \textit{finite} if $S$ is finite and it is \textit{invertible} if, for all $s\in
S$, the transformation $\mu(s, \cdot):X\rightarrow X$ is a permutation of $X$. An automaton $\mathcal{A}$ can be visually
represented by its \textit{Moore diagram}: this is a directed labeled graph whose vertices are identified with the states of
$\mathcal{A}$. For every state $s\in S$ and every letter $x\in X$, the diagram has an arrow from $s$ to $\lambda(s,x)$
labeled by $x|\mu(s,x)$. A sink $id$ in $\mathcal{A}$ is a state with the property that $\lambda(id,x)=id$ and $\mu(id,x)=x$ for any $x\in X$.

An important class of automata is given by the so-called bounded automata \cite{sidki}. An automaton is said to be \emph{bounded} if the sequence of numbers of paths of length $n$ avoiding the sink state (along the directed edges of the Moore diagram) is bounded.

For each $n\geq 1$, let $X^n$ denote the set of words of length $n$ over the alphabet $X$ and put $X^0  = \{\emptyset\}$, where $\emptyset$ is the empty word. The action of $\mathcal{A}$ can be easily extended to the infinite set $X^\ast= \bigcup_{n=0}^\infty X^n$ as follows:
\begin{eqnarray}\label{actionextended}
\lambda(s,xw) = \lambda(\lambda(s,x),w) \qquad \qquad \mu(s,xw) = \mu(s,x)\mu(\lambda(s,x),w),
\end{eqnarray}
for every $w\in X^\ast$. For a state $s\in S$, we denote by $\mathcal{A}_s$ the transformation $\mu(s,\cdot)$ on $X^{\ast}$ defined by Eqs. \eqref{actionextended}, which is a bijection if $\mathcal{A}$ is invertible.
Given the invertible automaton $\mathcal{A}$, the \textit{automaton group} generated by $\mathcal{A}$ is by definition the group generated by the transformations $\mathcal{A}_s$, for $s\in S$, and it is denoted $G(\mathcal{A})$. In the rest of the paper, we will often use the notation $s$ instead of $\mathcal{A}_s$. Moreover, the maps $\lambda$ and $\mu$ can be naturally extended to each element of $G(\mathcal{A})$. Notice that the action of $G(\mathcal{A})$ on $X^\ast$ preserves the sets $X^n$, for each $n$.\\
\indent It is a remarkable fact that an automaton group can be regarded in a very natural way as a group of automorphisms of the regular rooted tree of degree $|X|=k$, i.e., the rooted tree $T_k$ in which each vertex has $k$ children, via the identification of the $k^n$ vertices of the $n$-th level of $T_k$ with the set $X^n$.\\
\indent The group $G(\mathcal{A})$  is said to be \textit{spherically transitive} if its action is transitive on $X^n$, for any $n$. Let $g\in G(\mathcal{A})$. The action of $g$ on $X^\ast$ can be factorized by considering the action on $X$ and $|X|$ restrictions as follows. Let $Sym(k)$ be the symmetric group on $k$ elements. Then an element $g\in G(\mathcal{A})$ can be represented as
\begin{eqnarray}\label{ssd}
(g_1,\ldots, g_{k})\sigma,
\end{eqnarray}
where $g_i:=\lambda(g,i)\in G(\mathcal{A})$ and $\sigma\in Sym(k)$ describes the action of $g$ on $X$. We say that Eq. \eqref{ssd} is the \emph{self-similar representation} of $g$. Notice that, given $g=(g_1,\ldots, g_{k})\sigma$ and $h=(h_1,\ldots, h_{k})\tau$, the self-similar representation of $gh$ is
$$
gh=(g_1h_{\sigma(1)}, \ldots, g_{k}h_{\sigma(k)})\sigma\tau.
$$
In the tree interpretation of Eq. \eqref{ssd}, the permutation $\sigma$ corresponds to the action of $g$ on the first level of $T_k$, and the automorphism $g_i$ is the restriction of the action of $g$ to the subtree (isomorphic to $T_k$) rooted at the $i$-th vertex of the first level.

Let us denote by $Stab_{G(\mathcal{A})}(X) = \{g\in G(\mathcal{A}): g(x) = x, \ \forall \ x\in X\}$ the stabilizer of $X$, that is, the subgroup of $G(\mathcal{A})$ consisting of all elements acting trivially on $X$. In particular, if $g\in Stab_{G(\mathcal{A})}(X)$, then one can identify $g$ with the $k$-tuple $(g_1,\ldots, g_{k})$, where $g_i=\lambda(g,i)$.

\begin{lemma}\label{lemrest}
Let $g=(g_1,\ldots, g_k) \in Stab_{G(\mathcal{A})}(X)$. If $g_i \in \{g,id\}$ for each $i=1,\ldots, k$, then $g$ is the trivial element of $G(\mathcal{A})$.
\end{lemma}
\begin{proof}
Let $w\in X^n$. We will prove by induction on $n$ that $g(w) = w$. The case $n=1$ is obvious since $g \in Stab_{G(\mathcal{A})}(X)$. Now let $w = xw'$, with $x\in X$ and $w'\in X^{n-1}$. Then one has $g(xw') = x g_x(w')$. If $g_x=id$, the claim easily follows. If $g_x=g$, the inductive hypothesis gives $g_x(w') = w'$ and so one gets $g(w)=w$.
\end{proof}
Given two groups $G_1$ and $G_2$, the direct product of $G_1$ and $G_2$ will be denoted by $G_1 \times G_2$. In particular, we denote by $G^k$ the $k$-times iterated direct product of a group $G$ with itself. The following definitions are given in the literature for the broader class of self-similar groups (see \cite{nekrashevyc}). Here we restrict our interest to automaton groups.

\begin{definition}
Let $G(\mathcal{A})$ be an automaton group. Then $G(\mathcal{A})$ is said to be
\begin{enumerate}
\item \emph{fractal}, if the map $\psi: Stab_{G(\mathcal{A})}(X)\to G(\mathcal{A})^k$ given by $g\mapsto (g_1,\ldots, g_{k})$ is surjective on each factor.
\item \emph{weakly regular branch} over its subgroup $N$, if $N^k\subset\psi(N\cap Stab_{G(\mathcal{A})}(X))$, where $N$ is supposed to be nontrivial.
\end{enumerate}
\end{definition}
 Let $\mathcal{G}$ be a finitely generated group, with symmetric generating set $S=S^{-1}$. The \textit{length} of $g$ with respect to $S$ is denoted by $|g|_S$ (we will omit the subscript $S$ when the generating set is fixed) and it is defined as the minimal length of a word in $S^\ast$ representing the element $g$. Then one puts $\gamma_S(n)= \# \{g\in \mathcal{G} : |g|_S \leq n\}$. This defines the growth function
$$
\gamma_S:\mathbb{N}\to \mathbb{N}.
$$
This map clearly depends on the generating set $S$. However, one can prove that changing the generating set does not affect the asymptotic properties of $\gamma_S$. In particular, the growth rate
$$
\lambda_{\mathcal{G}}=\lim_{n\to\infty} \sqrt[n]{\gamma_S(n)}
$$
is greater than or equal to $1$ independently on the particular generating set (see, for instance, \cite{pierre}).
\begin{definition}
A finitely generated group $\mathcal{G}$ has \emph{exponential growth} (resp. \emph{sub-exponential growth}) if $\lambda_{\mathcal{G}}>1$ (resp. $\lambda_{\mathcal{G}}=1$).
\end{definition}

\section{Automata groups from graphs}
We present in this section the main construction of the paper, that is, we are going  to associate an invertible automaton with a given finite graph.\\
\indent Let $G=(V,E)$ be a finite graph. Let $V=\{x_1,\ldots, x_k\}$ (we will often use also the notation $\{1,2,\ldots, k\}$) be its vertex set and let $E$ be its edge set.\\
\indent First of all, we choose an orientation for the edges of $G$. Let $E'$ be the set of edges, where an orientation of each edge has been chosen. Notice that elements in $E$ are unordered pairs of type $\{x_i,x_j\}$, whereas elements in $E'$ are ordered pairs of type $(x_i,x_j)$, meaning that the edge has been oriented from the vertex $x_i$ to the vertex $x_j$.\\
\indent We then define an automaton $\mathcal{A}_G=(E' \cup \{id\}, V, \lambda, \mu)$ such that:
\begin{itemize}
\item $E' \cup \{id\}$ is the set of states;
\item $V$ is the alphabet;
\item $\lambda: E'\times V\to E'$ is such that, for each $e=(x,y)\in E'$, one has
$$
\lambda (e,z) = \left\{
                  \begin{array}{ll}
                    e & \hbox{if } z=x \\
                    id & \hbox{if } z\neq x;
                  \end{array}
                \right.
$$
\item $\mu: E'\times V\to V$ is such that, for each $e=(x,y)\in E'$, one has
$$
\mu (e,z) = \left\{
                  \begin{array}{ll}
                    y & \hbox{if } z=x \\
                    x & \hbox{if } z=y \\
                    z & \hbox{if } z\neq x,y.
                  \end{array}
                \right.
$$
\end{itemize}
In other words, any directed edge $e=(x,y)$ represents a state of the automaton $\mathcal{A}_G$ and it has just one restriction to itself (given by $\lambda(e,x)$) and all other restrictions to the sink $id$. Its action is nontrivial only on the letters $x$ and $y$, which are switched since $\mu(e,x)=y$ and $\mu(e,y)=x$. It is easy to check that
$\mathcal{A}_G$ is invertible for any $G$ and any choice of the orientation of the edges. This makes us able to define an associated automaton group.

\begin{definition}\label{maindefi}
The graph automaton group $\mathcal{G}_G$ is the automaton group generated by $\mathcal{A}_G$.
\end{definition}

\begin{rem}\rm
Notice that any loop in $G$ gives rise to the trivial element of $\mathcal{G}_G$. Moreover, any multiedge produces a set of equal generators (up to consider the inverse).
\end{rem}

The previous remark basically says that we can just consider simple graphs. We make from now on this assumption. The following proposition shows that the graph automaton group $\mathcal{G}_G$ does not depend on the particular orientation of the edges in $G$ nor on the order of the vertices in $V$.

\begin{proposition}\label{lemma_iso}
The group $\mathcal{G}_G$ does not depend on the choice of the edge orientation. Moreover, a rearrangement of the vertices in $V$ produces groups that are isomorphic.
\end{proposition}
\begin{proof}
Without loss of generality, we can suppose that there exists an edge connecting the vertices $x_1$ and $x_2$. If we choose the orientation from $x_1$ to $x_2$, then the self-similar representation of the corresponding edge $e$ as a generator of $\mathcal{G}_G$ is
$$
e=(e,id,\ldots, id)(1,2).
$$
On the other hand, if the opposite orientation is chosen, the self-similar representation of the corresponding edge $f$ becomes
$$
f=(id,f,\ldots, id)(1,2).
$$
A direct computation gives $ef=(ef, id, \ldots, id)$, so that $ef=id$ by Lemma \ref{lemrest}, and so $f=e^{-1}$. Therefore, the choice of the orientation does not affect the structure of the group $\mathcal{G}_G$.\\
\indent In order to prove the second claim, it is enough to observe that changing the name of the vertices is equivalent to act by a permutation $\sigma$ on the alphabet $V$. In this case, we just get the group $\mathcal{G}_G^{\sigma}$, which is obtained from $\mathcal{G}_G$ via conjugation by $\sigma$.
\end{proof}
In the light of Proposition \ref{lemma_iso}, given an oriented edge $e = (x_i,x_j)\in E'$, we can denote by $e^{-1}=(x_j,x_i)$ the edge with the opposite orientation, keeping in mind that the choice of the opposite orientation corresponds to take the inverse in $\mathcal{G}_G$.
\begin{rem}  \rm
For any graph $G$ the automaton $\mathcal{A}_G$ is bounded. In fact, it is easy to check that the Moore diagram of $\mathcal{A}_G$, regarded as a simple graph, is a star graph with one internal vertex corresponding to the sink $id$ and $|E|$ leaves corresponding to the directed edges of $G$. Each leaf has exactly one directed loop and all other transitions go to the sink. In particular, for any edge of $G$, the Moore diagram contains exactly one cycle of fixed length avoiding the sink and disconnected from the other cycles. As a consequence, the group $\mathcal{G}_G$ does not contain free non abelian subgroups, as follows from Sidki's theorem \cite{sidki}. Moreover, since all groups generated by bounded automata are amenable \cite{amenability bounded}, then each graph automaton group $\mathcal{G}_G$ is amenable. We do not provide a special description of \emph{F\o lner sets} of $\mathcal{G}_G$. The matter certainly deserves further investigation in the future. Notice that, having $\mathcal{G}_G$ \emph{solvable world problem}, a sequence of F\o lner sets must be computable (in the sense of \cite{cava1,cava2}).
\end{rem}
The following result shows that taking subgraphs in $G$ corresponds to obtain subgroups in $\mathcal{G}_G$. By a subgraph of $G$, here we mean a subset of its edges, together with the subset of vertices of $V$ consisting of their endpoints.

\begin{proposition}\label{prop_subgroup}
Let $H=(V_H,E_H)$ be a graph isomorphic to a subgraph of $G$. Then $\mathcal{G}_H \leq \mathcal{G}_G$.
\end{proposition}
\begin{proof}
Let $e_1,\ldots, e_{\ell}$ be the edges of $G=(V,E)$ belonging to the subgraph $\widetilde{G}$ isomorphic to $H$. Consider the subgroup $ \widetilde{\mathcal{G}}\leq\mathcal{G}_G $ generated by $e_1, \ldots, e_{\ell}$. We claim that the groups $\widetilde{\mathcal{G}}$ and $\mathcal{G}_H$ are isomorphic. Let $\psi$ be the isomorphism between $\widetilde{G}$ and $H$. By Proposition \ref{lemma_iso} we can suppose, without loss of generality, that $e_1,\ldots, e_{\ell}$ are edges of $G$, whose endpoints are the first $t$ vertices of $V$, and we can consider in $V_H$ the order induced by $\psi$. If we focus on the self-similar representation of the automorphisms generating $\mathcal{G}_H$, this is exactly the self-similar representation of the $e_i$'s, regarded as generators of $\mathcal{G}_G$, where the last $|V|-t$ restrictions are erased. Therefore, it is clear that a relation in $ \widetilde{\mathcal{G}}$ holds if and only if the same relation in $ \mathcal{G}_H$ holds.
\end{proof}

\begin{proposition}\label{proppp}
Let $G=(V,E)$ be the disjoint union of the graphs $G_1=(V_1,E_1)$, $\ldots$, $G_t=(V_t,E_t)$. Then $ \mathcal{G}_G = \mathcal{G}_{G_1}\times \cdots\times \mathcal{G}_{G_t}$.
\end{proposition}
\begin{proof}
The claim easily follows by observing that the generators in $E_i$ act trivially on the set of vertices $V\setminus V_i$.
\end{proof}

By virtue of Proposition \ref{proppp}, we can suppose $G$ to be connected. Therefore, from now on, we assume $G=(V,E)$ to be simple and connected.

\begin{example} \label{esempibase}    \rm
\begin{enumerate}
\item If $G$ is the path graph $P_2$ on $2$ vertices, the associated automaton $\mathcal{A}_{P_2}$ is represented in Fig. \ref{automatonP2}.
\begin{figure}[h]
\begin{center}
\psfrag{a}{$a$}\psfrag{1}{$1$}\psfrag{2}{$2$}\psfrag{id}{$id$}\psfrag{1|1,2|2}{$1|1, 2|2$}\psfrag{1|2}{$1|2$}\psfrag{2|1}{$2|1$}
\includegraphics[width=0.55\textwidth]{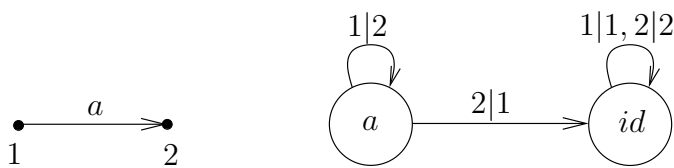}
\end{center}\caption{The path $P_2$ and the associated automaton $\mathcal{A}_{P_2}$.} \label{automatonP2}
\end{figure}
The automaton $\mathcal{A}_{P_2}$ is the so-called \emph{Adding machine} (see, e.g., \cite{fractal_gr_sets}) and it generates the group $\mathbb{Z}$. The self-similar representation of its generator is
$$
a=(a,id)(1,2).
$$
\item If $G$ is the path graph $P_3$ on $3$ vertices, the associated automaton $\mathcal{A}_{P_3}$ is represented in Fig. \ref{automatonP3}.
\begin{figure}[h]
\begin{center}
\psfrag{a}{$a$}\psfrag{b}{$b$}\psfrag{1}{$1$}\psfrag{2}{$2$}\psfrag{3}{$3$}\psfrag{id}{$id$}\psfrag{1|2}{$1|2$}\psfrag{2|3}{$2|3$}\psfrag{2|1,3|3}{$2|1, 3|3$}
\psfrag{1|1,3|2}{$1|1, 3|2$}\psfrag{1|1,2|2,3|3}{$1|1, 2|2, 3|3$}
\includegraphics[width=0.55\textwidth]{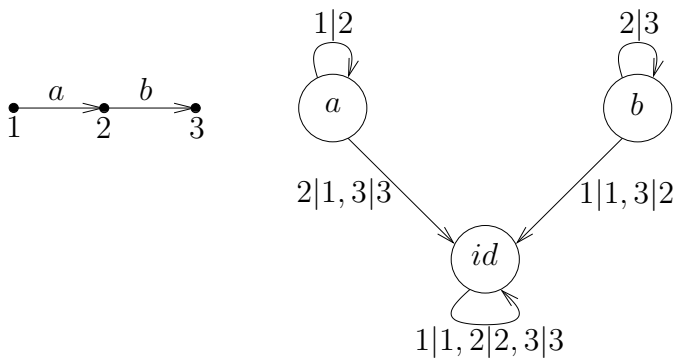}
\end{center}\caption{The path $P_3$ and the associated automaton $\mathcal{A}_{P_3}$.} \label{automatonP3}
\end{figure}
It corresponds to the so-called \emph{Tangled odometer} (see, e.g., \cite{spettro}). The two generators of the group $\mathcal{G}_{P_3}$ have the self-similar representation:
$$
a=(a,id,id)(1,2) \qquad b=(id,b,id)(2,3).
$$
\item If $G$ is the cyclic graph $C_3$ on $3$ vertices, the associated automaton $\mathcal{A}_{C_3}$ is represented in Fig. \ref{automatontri}.
\begin{figure}[h]
\begin{center}
\psfrag{a}{$a$}\psfrag{b}{$b$}\psfrag{c}{$c$}\psfrag{1}{$1$}\psfrag{2}{$2$}\psfrag{3}{$3$}\psfrag{id}{$id$}\psfrag{1|2}{$1|2$}\psfrag{2|3}{$2|3$}\psfrag{2|1,3|3}{$2|1, 3|3$}
\psfrag{1|1,3|2}{$1|1, 3|2$}\psfrag{1|1,2|2,3|3}{$1|1, 2|2, 3|3$}\psfrag{1|3,2|2}{$1|3, 2|2$}\psfrag{3|1}{$3|1$}
\includegraphics[width=0.6\textwidth]{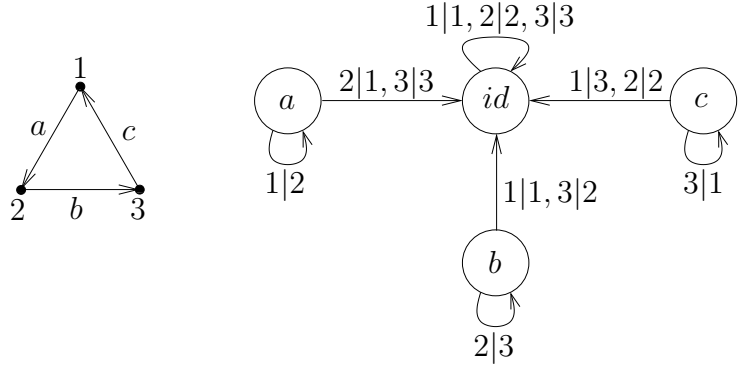}
\end{center}\caption{The cycle $C_3$ and the associated automaton $\mathcal{A}_{C_3}$.} \label{automatontri}
\end{figure}
The automaton $\mathcal{A}_{C_3}$ generates a group that we call the \emph{uncle of the Hanoi Towers group} on three pegs (see, e.g., \cite{hanoi}), and the three generators of the group $\mathcal{G}_{C_3}$ have the self-similar representation:
$$
a=(a,id,id)(1,2) \qquad  b=(id,b,id)(2,3) \qquad  c=(id,id,c)(3,1).
$$
Observe that the automorphisms $A:=abc$, $B:=bca$, and $C:=cab$ generate the classical Hanoi Towers group on three pegs, that is therefore a subgroup of its uncle.
\end{enumerate}
\end{example}

Notice that if $e,f\in E'$ do not share any vertex, then $[e,f]=e^{-1}f^{-1}ef=id$ in $\mathcal{G}_G$, since their actions are nontrivial on disjoint subsets of $V$. The next theorem collects some more algebraic properties shared by (almost) all graph automaton groups. \\
\indent A directed path of length $t$ from the vertex $v$ to the vertex $w$ in a graph $G=(V,E)$ is a sequence of vertices $v_{i_0}=v, v_{i_1}, v_{i_2}, \ldots, v_{i_t}=w$ in which all edges $e_1, e_2, \ldots, e_t$ are oriented in the direction from $v$ to $w$, that is, one has $e_1=(v,v_{i_1}), e_2=(v_{i_1},v_{i_2}), \ldots, e_t=(v_{i_{t-1}},w)$. A directed cycle of length $t$ is a directed path such that $v_{i_0}=v_{i_t}$.

\begin{theorem}\label{teo_principale}
Let $G=(V,E)$ be a graph such that $|E|\geq 2$. Then the following properties hold.
\begin{enumerate}
\item $\mathcal{G}_G$ is fractal.
\item $\mathcal{G}_G$ contains an element of finite order and is non-abelian.
\item If $G$ contains a cycle $e_1 ,\ldots, e_t$, then $(e_1^{\varepsilon_1}\cdots e_t^{\varepsilon_t})^{t-1}$, with $\varepsilon_i\in \{\pm 1\}$, is a relation in $\mathcal{G}_G$ whenever $e_1^{\varepsilon_1},\ldots, e_t^{\varepsilon_t}$ is a directed cycle in $G$.
\item $\mathcal{G}_G$ is weakly regular branch over its commutator subgroup $ \mathcal{G}_G'$.
\end{enumerate}
\end{theorem}
\begin{proof}
\begin{enumerate}
\item We have to show that  the map $\psi: Stab_{\mathcal{G}_G}(V)\to \mathcal{G}_G^k$ given by $g\mapsto (g_1,\ldots, g_{k})$, with $g_i=\lambda(g,x_i)$, is surjective on each factor. Let $e$ be an edge oriented from the vertex $x_i$ to the vertex $x_j$ (suppose $i < j$). In order to prove fractalness, it is enough to produce an element of $Stab_{\mathcal{G}_G}(V)$ whose $h$-th restriction is $e$, for each $h=1,\ldots, k$. An explicit computation gives:
$$
e^2=(id, \ldots, id,\underbrace{ e}_{i\textrm{-th place}}, id, \ldots, id, \underbrace{ e}_{j\textrm{-th place}},id,\ldots, id)\in Stab_{\mathcal{G}_G}(V).
$$
Therefore, we can focus on the case where $h\neq i,j$. Consider a path of length $t$ in $G$ from the vertex $x_i$ to the vertex $x_h$ through the vertices $x_i=x_{i_0},x_{i_1},\ldots, x_{i_{t-1}}, x_{i_t}=x_h$. Suppose that the path passes along the edges $e_{i_1},\ldots, e_{i_t}$, where the endpoints of $e_{i_l}$ are $x_{i_{l-1}},x_{i_l}$, for each $l=1,\ldots, t$. Take the group element $g:=e_{i_1}^{\varepsilon_1}\cdots e_{i_t}^{\varepsilon_t}$, where $\varepsilon_{s}= 1$ if in the path from $x_i$ to $x_h$ the edge $ e_{i_s}$ is oriented in the direction of the path, and  $ \varepsilon_{s}=-1$ otherwise. One can check that
$$
g=(id, \ldots, id, \underbrace{ g}_{i\textrm{-th place}},id,\ldots, id)\sigma
$$
where $\sigma$ is a cyclic permutation of length $t+1$ such that $\sigma(i_l)=i_{l-1}$ for any $l=1,\ldots, t$ and $\sigma(i)=h$. A direct computation gives
$$
g^{t+1}=(id, \ldots, id, \underbrace{ g}_{i\textrm{-th place}},\ldots, \underbrace{ g}_{i_1\textrm{-th place}}, \ldots, \underbrace{ g}_{h\textrm{-th place}},id, \ldots, id),
$$
i.e., the nontrivial restrictions in the self-similar representation of $g^{t+1}$ coincide with $g$ and correspond to the positions $i=i_0,i_1,\ldots, i_t=h$. Moreover, notice that $g^{t+1}\in Stab_{\mathcal{G}_G}(V)$. Being $Stab_{\mathcal{G}_G}(V)$ a normal subgroup of $\mathcal{G}_G$, one has that $g^{-1}e^2g\in Stab_{\mathcal{G}_G}(V)$ and it is easy to check that
$$
\lambda(g^{-1}e^2g,x_{i_t})=g^{-1}eg.
$$
Then we obtain
$$
\lambda(g^{t+1}g^{-1}e^2gg^{-(t+1)}, x_{i_t})=e,
$$
i.e., we have moved the generator $e$ to the $h$-th restriction in the self-similar representation of $g^{t+1}g^{-1}e^2gg^{-(t+1)}$. The same method can be applied to any generator and to any $h$, and this concludes the proof.
\item Let $e=(x_i,x_j)$ and $f=(x_j,x_h)$, so that $e$ and $f$ share the vertex $x_j$, and their self-similar representations as generators of ${\mathcal{G}_G}$ are
\begin{eqnarray*}
\hspace{1.5cm} e=(id, \ldots, id, \underbrace{ e}_{i\textrm{-th place}},id,\ldots, id)(i,j) \quad  f=(id, \ldots, id, \underbrace{ f}_{j\textrm{-th place}},id,\ldots, id)(j,h).
\end{eqnarray*}
A direct computation gives
$$
[e,f]=(id, \ldots, id,\underbrace{ f}_{j\textrm{-th place}}, id, \ldots, id, \underbrace{f^{-1} }_{h\textrm{-th place}},id,\ldots, id)(i,j,h).
$$
Therefore $[e,f]\neq id$ but  $[e,f]^3=id$.
\item Up to change the orientation of some edges (see Proposition \ref{lemma_iso}), we can assume that $e_1 ,\ldots, e_t$ is a directed cycle centered at $x_i$, i.e., a directed closed path with all edges oriented in the same direction. Suppose that the cycle contains the vertices $x_i = x_{j_0},\ldots, x_{j_t}=x_i$. In particular, $e_i$ corresponds to the directed edge $(x_{j_{i-1}}, x_{j_{i}})$. A direct computation gives
$$
e_1\cdots e_t=(id, \ldots, id,\underbrace{ e_1\cdots e_t}_{i\textrm{-th place}}, id, \ldots, id)\sigma
$$
where $\sigma$ is a cyclic permutation of length $t-1$ such that $\sigma(i)=i$. In particular, $\sigma^{t-1}=id$ and so
$$
 (e_1\cdots e_t)^{t-1}=(id, \ldots, id,\underbrace{ (e_1\cdots e_t)^{t-1}}_{i\textrm{-th place}}, id, \ldots, id)
$$
and so $(e_1\cdots e_t)^{t-1}=id$ by Lemma \ref{lemrest}.
\item We have to prove that, given any $k$-tuple $(g_1, \ldots, g_k)$, with $g_i \in \mathcal{G}_G'$ for each $i$, there exists an element $g \in \mathcal{G}_G' \cap Stab_{\mathcal{G}_G}(V)$ such that $\psi (g) =  (g_1, \ldots, g_k)$. We write $\mathcal{G}_G'>\mathcal{G}_G'\times \cdots \times \mathcal{G}_G'$ to denote this condition. First, recall that if $e$ and $f$ do not share any vertex, then they commute in $\mathcal{G}_G$, since they act nontrivially on sets of letters which are disjoint.

It follows that $\mathcal{G}_G'$ is normally generated by the commutators $[e,f]$ such that $e$ and $f$ share a vertex. So let $e,f$ be generators of $\mathcal{G}_G$ such that $e=(x_i,x_j)$ and $f=(x_j,x_k)$. We can suppose, without loss of generality, that $i=1, j=2$ and $k=3$. In particular:
$$
e^2=(e,e,id,\ldots, id) \qquad  f^2=(id,f,f,id,\ldots, id).
$$
A direct computation gives
$$
[e^2,f^2]= (id, [e,f], id,\ldots, id).
$$
By proceeding as in the proof of Claim (1), we get that, given an index $h\neq 2$, it is possible to construct an element $g$ such that
$$
g^{-1}[e^2,f^2]g=(id, \ldots, id,\underbrace{[e,f]}_{h\textrm{-th place}},id,\ldots, id)\in \mathcal{G}_G'.
$$
Then, by using that $\mathcal{G}_G'$ is normal and that  $\mathcal{G}_G$ is fractal, we get
$$
(id,\ldots ,id ,\underbrace{ [e,f]^{\mathcal{G}_G}}_{h\textrm{-th place}},id, \ldots,  id) \subseteq \mathcal{G}_G'.
$$
Being  $h$ arbitrary and by applying  the same argument to any pair of generators, we get
$$
\mathcal{G}_G'>\mathcal{G}_G'\times \cdots \times \mathcal{G}_G',
$$
which is our claim.
\end{enumerate}
\end{proof}
If we consider the Case (1) in Example \ref{esempibase}, where the initial graph contains only one edge, it gives rise to the Adding machine. This group is fractal but it is abelian, free and not weakly regular branch. Basically, this is the only nontrivial case for which the properties (2)-(4) of Theorem \ref{teo_principale} do not hold.\\

Let us focus now on the semigroup structure of graph automaton groups. This will allow to know the growth of graph automaton groups (see Corollary \ref{corogrowthexp}).

\begin{theorem}\label{teo_esp}
Let $G=(V,E)$ be a graph such that $|E|\geq 2$. Let $e,f$ be edges that share a vertex in $G$. Then the semigroup generated by $e$ and $f$ is free.
\end{theorem}
\begin{proof}
By virtue of Proposition \ref{prop_subgroup}, it is enough to consider the group $H$ generated by the elements
$$
e=(e,id,id)(1,2) \qquad f=(id,id ,f)(2,3),
$$
since it will be a subgroup of any group $\mathcal{G}_G$ associated with a connected graph $G$ with more than one edge. If the semigroup $\mathcal{S}_{e,f}$ is free, then we are done.
In particular, we have to prove that if $u\neq v$ in $\{e,f\}^\ast$, then it must be $u\neq v$ in $\mathcal{S}_{e,f}$ (notice that, in the notation of Case (2) of Example \ref{esempibase}, we are going to show that $\mathcal{S}_{a,b^{-1}}$ is free).\\
\indent Observe that, by using the self-similar representations of $e$ and $f$, we get:
$$
e^2=(e,e,id),\quad \ ef=(e,id,f)(1,3,2),\quad \ fe=(e,id,f)(1,2,3),\quad \ f^2=(id,f,f).
$$
This implies that any word in $\{e,f\}^\ast$ of length greater than $1$ has restrictions of shorter length. Moreover, if $u=(u_1,u_2,u_3)\sigma$ and $v=(v_1,v_2,v_3)\tau$, then a relation $u=v$ corresponds to the permutation equality $\sigma=\tau$ and to the restriction equality $u_i=v_i$ in $\mathcal{S}_{e,f}$, for each $i=1,2,3$. In particular, if at least one between $u,v$ has length greater than $1$, this would give rise to relations $u_i=v_i$ such that $|u_i|+|v_i|<|u|+|v|$, for $i=1,2,3$.\\
\indent Now let $u=v$ be a relation in the semigroup with smallest length $|u|+|v|$ in $\{e,f\}^\ast$. By the cancellativity of the semigroup, we may assume that the words $u$ and $v$ do not start and end with the same letter; in particular, we can suppose $u=eu'$ and $v=fv'$.

Since we have supposed that $u=v$ is a relation in the semigroup with smallest length $|u|+|v|$, then $u$ and $v$ must have the same restrictions (as words in $\{e,f\}^\ast$, and not only in $\mathcal{S}_{e,f}$) at each position, otherwise, by considering restrictions, we would get relations of a shorter length.
In what follows, we show that in all cases in which $u$ and $v$ do not coincide, there exists some restriction in which they differ, and this contradicts minimality.\\
\indent If $u=e^{\ell}$, and $v=f^nv'$ (with $\ell, n \geq 1$), then by considering the restriction to the position $3$ we get the equation $id=fv''$, which is a contradiction by the minimality of the relation $u=v$. By using a symmetric argument, we may assume
\[
u=e^{m}f^{k}z, \quad v=f^{n}e^{t}y
\]
for some $m,k,n,t\geq 1$.\\
\indent Let $m,n\geq 2$. By considering the restriction to the position $2$ we get $ew= f w'$, which is a contradiction again.\\
\indent Therefore, without loss of generality, we may assume $m=1$. Suppose now that $n\geq 2$. We consider two cases: either $z$ is the empty word or not. In the first case, the restriction to the position $2$ gives $id=fw$, and this is a contradiction for the same reason as above. If $z$ is not the empty word, we necessarily have $z=ez'$ and in this case, looking at the restriction to the position $2$, we get $ez''=fw$, a contradiction again.\\
\indent Thus, we deduce that it must be $m=n=1$, i.e., $u=ef^kz$, $v=fe^ty$. If $z,y$ are nonempty, then necessarily we have $u=ef^kez'$, $v=fe^tfy'$. Now, by restricting to the position $2$, we get $ez''=fy''$, a contradiction. Thus, without loss of generality, we may assume $z=\emptyset$ so that $u=ef^k$.\\
\indent Now if $y$ is nonempty and $v=fe^tfy'$, by considering the restriction to the position $2$, we get $id=fy''$, a contradiction. Therefore, we reduce to the case $z=y=\emptyset$, that is, we can assume $u=ef^{k}, v=fe^{t}$. Since we have $\mu(ef^{k},2)= 1 \neq 3 = \mu(fe^t,2)$, we have a contradiction and the proof is completed.
\end{proof}

\begin{corollary}\label{corogrowthexp}
Let $G=(V,E)$ be a graph such that $|E|\geq 2$. Then $\mathcal{G}_G$  has exponential growth.
\end{corollary}
\begin{proof}
It follows from the fact that $\mathcal{G}_G$ contains a free semigroup of two generators $a,b$. In fact, in this case, with any generating set containing $a$ and $b$ we can construct at least $2^n$ distinct group elements of length $n$.
\end{proof}

\begin{rem}\rm
Notice that, for the notion of semigroup, the chosen orientation of the edges is important. Observe that the Adding machine, which is isomorphic to the infinite cyclic group $\mathbb{Z}$, has polynomial growth; on the other hand, the associated graph $G$ is the path $P_2$ on two vertices (see Case (1) in Example \ref{esempibase}), which does not satisfy the hypothesis of Theorem \ref{teo_esp} and Corollary \ref{corogrowthexp}.
\end{rem}

\subsection{A connection to right-angled Artin groups}
Let $G=(V,E)$ be a simple graph. One can construct a group associated with such a graph in the following way: the vertex set $V$ is the generating set and the only relations are given by the commutators of adjacent vertices. More precisely, given $G=(V,E)$, the group with presentation
$$
W(G) = \langle v\in V| vu=uv \textrm{ if } \{u,v\}\in E \rangle
$$
is the associated \emph{right-angled Artin group}. For more details about the theory, the reader is referred to \cite{artin}.

Given a graph  $G=(V,E)$, one can define its dual graph $G'$ to be the graph with vertex set $E$, where $e$ and $f$ are adjacent in $G'$ if they share a common vertex $v$ in $G$.\\
\indent Moreover, one can construct the complement $\overline{G}$ of $G$ having the same vertex set $V$, and where two vertices are adjacent if and only if they are not adjacent in $G$.

Let $G=(V,E)$ be a graph. The following proposition shows that there exists a relation between the groups $\mathcal{G}_G$ and $W(\overline{G'})$.

\begin{proposition}\label{artino}
Let $G=(V,E)$ be a simple graph. Then there exists an epimorphism $\phi:W(\overline{G'})\to \mathcal{G}_G$.
\end{proposition}
\begin{proof}
First of all notice that the generating set of $W(\overline{G'})$ is precisely $E$ (up to consider inverses). Let $e\in E$, we define $\phi(e)=e$, where $e$ is supposed to be oriented in $G$. The map $\phi$ is a well defined homomorphism. Moreover, the set of adjacent vertices  in $\overline{G'}$ corresponds exactly to those edges in $G$ that do not share any common vertex. These edges commute by definition, when regarded as elements of $\mathcal{G}_G$. In this way, we have that the set of relations in $W(\overline{G'})$ is contained in the set of relations of $\mathcal{G}_G$. This concludes the proof.
\end{proof}

\section{Schreier graphs}\label{sezioneschreier}
In this section we recall the notion of Schreier graphs and we study some properties of them in the context of graph automaton groups.\\
\indent Let $ \mathcal{G}$ be a finitely generated group with a set $S$ of generators such that $id  \not\in S$ and $S =S^{-1}$,  and
suppose that $ \mathcal{G}$ acts on a set $M$. Then one can consider a graph $\Gamma(\mathcal{G}, S, M)$  with vertex set $M$, where two vertices $m, m'$
are joined by an edge if there exists $s \in S$ such that $s(m) = m'$. If this is the case, we label the edge from $m$ to $m'$ by $s$, and the edge from $m'$ to $m$ by $s^{-1}$. Equivalently, we can think that the same (undirected) edge is labeled by $s$ near $m$ and by $s^{-1}$ near $m'$.\\
\indent If the action of $\mathcal{G}$ on $M$ is transitive, then the graph $\Gamma( \mathcal{G}, S, M)$ is connected and corresponds to the classical notion of Schreier graph $\Gamma( \mathcal{G}, S, Stab_{ \mathcal{G}}(m))$ of the group $ \mathcal{G}$ with respect to the stabilizer subgroup $Stab_\mathcal{G}(m)$ for some (any) $m \in M$ (see \cite{nekrashevyc}).

\begin{definition}
Let $\mathcal{A} = (S,X,\lambda,\mu)$ be an invertible automaton and let $G(\mathcal{A})$ be the associated automaton group. The $n$-th Schreier graph $\Gamma_n=\Gamma_n( G(\mathcal{A}),S, X^n)$ is the Schreier graph given by the action of $G(\mathcal{A})$ over $X^n$, with respect to the generators given by $S$ and their inverses.
\end{definition}
Notice that, although Schreier graphs are defined as directed and labeled graphs, for our purposes we will often consider them as undirected and unlabeled graphs.

In what follows, we denote by $\Gamma_n^G$ the $n$-th Schreier graph of the graph automaton group $\mathcal{G}_G$. The vertex set of $\Gamma_n^G$ is identified with $V^n$, where $G=(V,E)$. Observe that $\Gamma_{1}^G$ coincides with $G$ up to remove loops and multi-edges from  $\Gamma_{1}^G$. Therefore, we can say that $\Gamma_{1}^G$ and $G$ coincide as simple graphs. Moreover, the Schreier graph $\Gamma_n^G$ is a regular graph of degree $2|E|$ by definition.

\begin{example}\rm
The Schreier graphs $\Gamma_n^G$, for $n=1,2,3,4$, of the \emph{Tangled odometer group} introduced in Example \ref{esempibase}, obtained when $G$ is the path $P_3$ on $3$ vertices, are shown in Fig. \ref{primitre} and Fig. \ref{quattro}. Notice that infinite Schreier graphs of the same group, with a different system of generators, have been classified in \cite{mio}.
\begin{figure}[h]
\begin{center}    \scriptsize
\psfrag{0}{$1$}\psfrag{1}{$2$}\psfrag{2}{$3$}
\psfrag{11}{$22$}\psfrag{00}{$11$}\psfrag{22}{$33$} \psfrag{02}{$13$}\psfrag{20}{$31$}
\psfrag{111}{$222$}\psfrag{000}{$111$}\psfrag{222}{$333$}
\psfrag{202}{$313$}\psfrag{020}{$131$}
\includegraphics[width=0.8\textwidth]{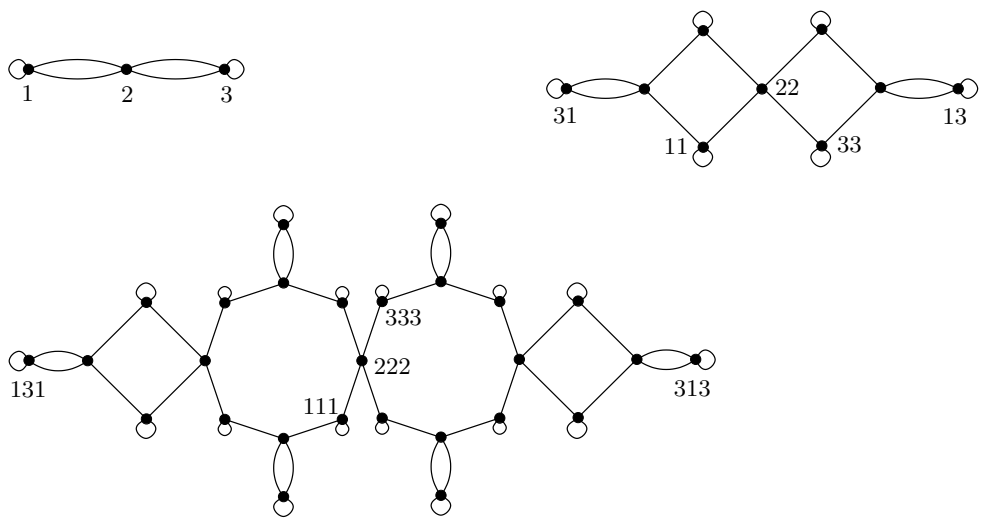}
\end{center}\caption{The Schreier graphs $\Gamma_1^{P_3}$, $\Gamma_2^{P_3}$, $\Gamma_3^{P_3}$ of the Tangled odometer group.} \label{primitre}
\end{figure}

\begin{figure}[h]
\begin{center}     \tiny
\psfrag{2020}{$3131$}\psfrag{0202}{$1313$}
\psfrag{1111}{$2222$}\psfrag{0000}{$1111$}\psfrag{2222}{$3333$}

\psfrag{2223}{$2223$}\psfrag{2113}{$2113$}\psfrag{2123}{$2123$}\psfrag{2213}{$2213$}\psfrag{2313}{$2313$}
\psfrag{3113}{$3113$}\psfrag{3123}{$3123$}

\includegraphics[width=0.9\textwidth]{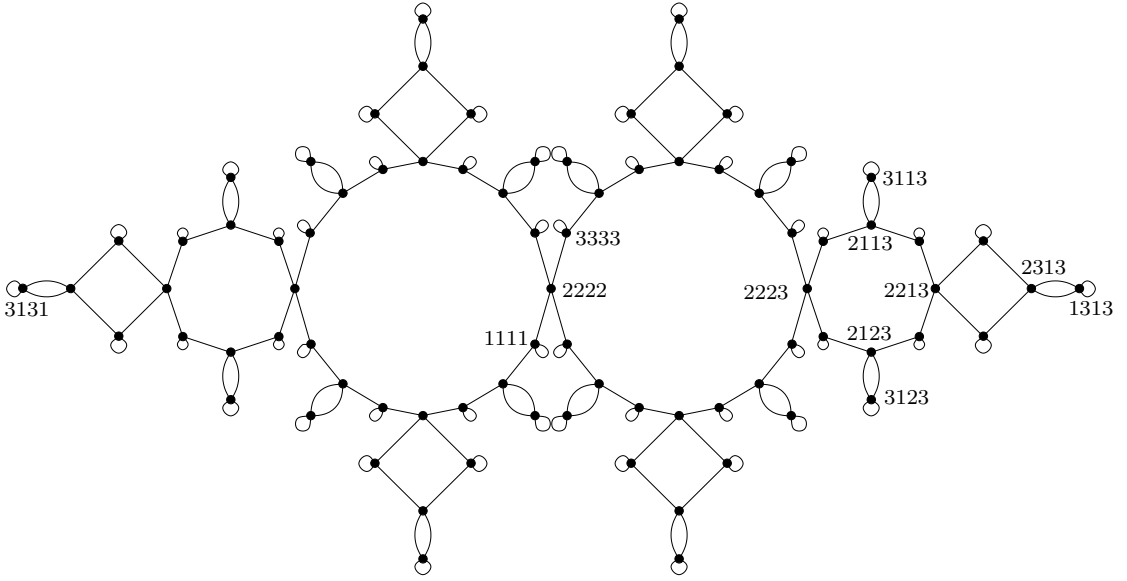}
\end{center}\caption{The Schreier graph $\Gamma_4^{P_3}$ of the Tangled odometer group.} \label{quattro}
\end{figure}
\end{example}

The Schreier graph $\Gamma_n^G$ only depends on the initial graph $G=(V,E)$ and not on the orientation of its edges, since the generating set that we let act on $V^n$ is supposed to be symmetric. The following lemma is well known \cite{fractal_gr_sets}.
\begin{lemma}\label{lemma_1}
$\Gamma_{n+1}^G$ is a covering of $\Gamma_{n}^G$ of degree $|V|$.
\end{lemma}
\begin{proof}[Sketch of the Proof]
The basic idea is that, if $vx$ and $wy$ are adjacent vertices in $\Gamma_{n+1}^G$, with $v,w\in V^n$ and $x,y\in V$, then $v$ and $w$ are adjacent in $\Gamma_{n}^G$.
\end{proof}

\begin{proposition}
For every $n\geq 1$, the graph $\Gamma_{n}^G$ is connected if and only if $G$ is. In particular, the group $\mathcal{G}_G$ is spherically transitive if and only if $G$ is connected.
\end{proposition}
\begin{proof}
Let us start by proving that, if $G$ is connected, then $\Gamma_{n}^G$ is connected for any $n\geq 1$. The connectedness of $G$ implies that $\mathcal{G}_G$ acts transitively on $V$. By Theorem \ref{teo_principale}, the group $\mathcal{G}_G$ is fractal and it is a standard inductive argument to show that these properties imply the transitivity of $\mathcal{G}_G$ on $V^n$.\\
\indent Conversely, if $x,y\in V$ are not in the same connected component of $G$, then there is no way to connect the vertices $x^n$ and $y^n$ in $\Gamma_{n}^G$.
\end{proof}
We are going to show how cut-vertices in $G$ propagate in the Schreier graphs $\Gamma_{n}^G$. Recall that a \emph{cut-vertex} of a graph is a vertex whose deletion increases the number of connected components of the graph (see, for instance, \cite{bollobas}).
Observe that, for our purposes, when we delete a cut-vertex from a graph we do not remove the edges which are incident to that vertex. We need some technical preparation.\\
\noindent Let $x$ be a cut-vertex of $G$, whose deletion disconnects $G$ into $c$ connected components $G_1, \ldots, G_c$ with $G_i=(V_i,E_i)$ for each $i=1,\ldots,c$. If a vertex $v$ of $V^n$ has the form $v=x^kyv'$, with $k\geq 1$, and $y \in V_i$, then a special subgraph of $\Gamma_{n}^G$ associated with the vertex $v$ can be constructed as follows.
\begin{itemize}
\item Let $E\setminus E_i$ act on $v$. Since only edges not belonging to $E_i$ are acting, this action can only change the prefix $x^k$ of $v$, but the suffix $yv'$ remains unchanged. Moreover, each edge $e\in E\setminus E_i$ generates an Adding machine, so that its action on $v$ produces an orbit which is a cyclic graph whose length is a power of $2$, and in particular it contains $v$ again. Let us denote by $X_{v,0}$ the orbit of $v$ under $E\setminus E_i$. Then put $Y_{v,0} =X_{v,0}\setminus \{v\}$.
\item Let $E_i$ act on $Y_{v,0}$ and get the set $X_{v,1}$. Concretely, we are appending new cycles of length a power of $2$ to the vertices contained in the cycles constructed at the previous step. Then put $Y_{v,1}= X_{v,1}\setminus \{v\}$.
\item Let $E\setminus E_i$ act $Y_{v,1}$ and get the set $X_{v,2}$. Then put $Y_{v,2}= X_{v,2}\setminus \{v\}$.
\end{itemize}
Continue in this way by alternating the action of generators in $E\setminus E_i$ and $E_i$; in this way, we construct an increasing sequence $Y_{v,m}\subseteq Y_{v,m+1}$. Since our alphabet is finite, after a finite number of steps, the sequence of sets $Y_{v,m}$ stabilizes to a set $Y_v$. Let $D_v$ be the graph induced by $Y_v$ (in particular, $D_v$ contains the vertex $v$ itself). We call the corresponding subgraph $D_v$ of $\Gamma_{n}^G$ the \emph{decoration} of $v$ in $\Gamma_{n}^G$.

\begin{rem}\label{AA}\rm
If $v=x^kyv'$ then the decoration $D_v$ is isomorphic to the subgraph of $\Gamma_{k}^G$ obtained by considering the alternate action of $E\setminus E_i$ and $E_i$ on $x^k$ as described above. In fact, in each vertex of $D_v$ the suffix $yv'$ remains unchanged. In particular, the isomorphism is obtained by deleting the last $n-k$ letters of each vertex of $D_v$. Moreover if the vertex $x$ is a leaf in $G$, that is, it has only one adjacent vertex, then it gives rise in $\Gamma_{n}^G$ to special cut-vertices which separate a component that is a loop.
\end{rem}

\begin{example}\rm
Consider the Schreier graph $\Gamma_4^{P_3}$ of Fig. \ref{quattro}, where $G=P_3$ is the path on $3$ vertices (as represented in Fig. \ref{automatonP3}). Take the vertex $v=2223$, where $x=2$ disconnects the path $P_3$ into the two components $(V_1,E_1) = (\{1\}, \{a\})$ and $(V_2,E_2)= (\{3\},\{b\})$. In particular, $y=3\in V_2$. Let us construct the decoration of $v$. We let the generator $a$ act on $v$ obtaining the $8$-cycle on the right of $v$. Now we let $b$ act on all vertices of this cycle different from $v$: we obtain two $2$-cycles attached to the vertices $2113$ and $2123$, a $4$-cycle attached to $2213$, together with four loops attached to the remaining vertices of the $8$-cycle. Then we let $a$ act again and we obtain: two loops attached to the vertices $3113$ and $3123$, two loops attached to two vertices of the $4$-cycle, and the $2$-cycle containing $2313$ and $1313$. We conclude by letting $b$ act, what produces the loop attached to the rightmost vertex $1313$.\\
\indent Roughly speaking, the decoration of $v$ is the subgraph of $\Gamma_4^{P_3}$ containing $v$ and all the vertices on the right of $v$. Notice that such a subgraph is isomorphic to the subgraph of $\Gamma_3^{P_3}$ (Fig. \ref{primitre}) obtained by taking the vertex $222$ and all the vertices on its left.
\end{example}

Notice that by construction every decoration $D_v$ corresponds to a subgraph that can be disconnected from the Schreier graph just by removing the vertex $v$. What said above can be summarized in the following result.

\begin{proposition}\label{cut}
Let $x$ be a cut-vertex for $G$. Then, for each $n\geq 2$, the vertex $v=xw\in V^n$ is a cut-vertex in $\Gamma_{n}^G$ for every $w\in V^{n-1}$, and it separates the decoration $D_v$ from the remaining part of the graph $\Gamma_n^G$.
\end{proposition}

\begin{rem} \rm
There exists a connection between our construction and a special class of graph products, which allows to give a purely combinatorial construction of the Schreier graphs of graph automaton groups. Actually, this description was our first attempt in defining the graphs $\Gamma_{n}^G$ without observing that $\Gamma_{n}^G$ arises from the action of an automaton group. The construction was inspired by the definition of Sierpi\'{n}ski graphs (see, for instance, \cite{sier} and bibliography therein), although in order to generate the latter by automata we should use a partial automaton.\\
\indent Let $G=(V,E)$ be a finite graph. One can define the $n$-th \emph{automaton power} $G^n_a$ of $G$ as the graph with vertex set $V^n$ and edge set consisting of pairs of vertices of type
\begin{equation}\label{regole}
\{x^tyw, y^txw \}  \qquad \{x^tzw,y^tzw\}   \qquad \{x^n,y^n\}
\end{equation}
where $\{x,y\}\in E$ and $z\neq x,y$, with $t=0,1,\ldots, n-1$ and $|w|=n-t-1$.\\
\indent It is easy to check that the edges described above are exactly the edges of $\Gamma_{n}^G$. In fact the connections described by (\ref{regole}) precisely correspond to the action of the states of the automaton $\mathcal{A}_G$ on $V^n$. Therefore, the graphs $\Gamma_{n}^G$ and $G^n_a$ are isomorphic.
\end{rem}

\subsection{Automorphisms of Schreier graphs of a graph automaton group}\label{subauto}

In this subsection we are going to investigate the relation between the automorphisms of the Schreier graph $\Gamma_{n}^G$ and the automorphisms of the initial graph $G=(V,E)$. Observe that Proposition \ref{cut} implies that any nontrivial automorphism of a decoration $D_v$ in $\Gamma_{n}^G$ is a nontrivial automorphism of $\Gamma_{n}^G$ fixing $v$. This observation will lead us to the description of the full automorphism group of the Schreier graphs associated with path graphs (Section \ref{paths}). We start with an extension result.
\begin{proposition}\label{ind}
Let $\phi$ be an automorphism of $G=(V,E)$, with $V=\{x_1, x_2, \ldots, x_k\}$. Then $\phi_n: \Gamma_{n}^G\to \Gamma_{n}^G$ defined by
$$
\phi_n(x_{i_1}\cdots x_{i_n})=\phi(x_{i_1})\cdots \phi(x_{i_n})
$$
is an automorphism of $\Gamma_{n}^G$.
\end{proposition}
\begin{proof}
First of all, notice that the map $\phi_n$ is a bijection by definition. We have to prove that, if $v$, $v'\in V^n$ are adjacent vertices in $\Gamma_{n}^G$, then $\phi_n(v)$ and $\phi_n(v')$ are adjacent too. \\
The adjacency in the graph $\Gamma_{n}^G$ are described by the rules from Eq. \eqref{regole}. We have three possibilities: either $v=x^tyw$ and  $v'=y^txw$, or $v=x^tzw$ and $v'=y^tzw$, or $v=x^n$ and $v'=y^n$, where $\{x,y\}\in E$ and $z\neq x,y$, with $t=0,1,\ldots, n-1$ and $|w|=n-t-1$. Since the vertices $v$ and $v'$ are supposed to be adjacent in $\Gamma_{n}^G$ then in all the three cases described above the vertices $x$ and $y$ are adjacent in $G$. This implies that $\phi(x)$ and $\phi(y)$ are adjacent in $G$, because $\phi$ is an automorphism of $G=(V,E)$. If $v=x^tyw$ and $v'=y^txw$, then
$$
\phi_n(v)=\phi(x)^t\phi(y)\phi_{n-t-1}(w), \ \ \phi_n(v')=\phi(y)^t\phi(x)\phi_{n-t-1}(w);
$$
it follows from Eq. \eqref{regole} that $\phi_n(v)$ and $\phi_n(v')$ are adjacent. The other two cases can be treated analogously.
\end{proof}
When the graph $G$ is cyclic, we can prove that actually all automorphisms of $\Gamma_{n}^G$ are of this type. Let us denote by $D_{2k}$ the dihedral group of $2k$ elements.

\begin{theorem}\label{cicli}
Let $C_k$ be the cyclic graph on $k$ vertices. Then
$$
Aut(\Gamma_{n}^{C_k})\cong Aut(C_k)\cong D_{2k}, \qquad \mbox{for each } n\geq 1.
$$ \end{theorem}
\begin{proof}
Observe that for $n=1$ the statement is obvious by definition of $\Gamma_{1}^{C_k}$. Therefore, we assume $n\geq 2$.
Let $V=\{1,\ldots, k\}$ be the vertex set of $C_k=(V,E)$. Observe that in $\Gamma_{n}^{C_k}$ one has $k^{n-1}$ sets of vertices of type $V_w^k=\{1w,\ldots, kw\}$, with $w\in V^{n-1}$, each yielding a copy of $C_k$ in $\Gamma_{n}^{C_k}$. This property can be obtained by taking $t=0$ in the rule $\{x^tyw, y^txw\}$ described in Eq. \eqref{regole}. Moreover, one has a copy of $C_k$ consisting of the vertices labeled by $\{1^n,2^n,\ldots, k^n\}$, according to the rule $\{x^n,y^n\}$ in Eq. \eqref{regole}.\\
\indent Let $\varphi$ be an automorphism of $\Gamma_{n}^{C_k}$. First of all, we want to prove that the set $X=\{1^n,\ldots, k^n\}$ is invariant under the action of $\varphi$. To prove that, we show that a vertex of $\Gamma_{n}^{C_k}$ is the unique common vertex of two copies of $C_k$ if and only if it belongs to $X$.

In order to avoid heavy notation, we omit here and in the sequel to specify every time that sums and differences must be taken modulo $k$. We denote by $e_i$ the generator associated with the edge $(i,i+1)$ of $C_k$.
It follows from Eq. \eqref{regole} that the neighbours of $i^n$ in $\Gamma_{n}^{C_k}$ are exactly the vertices $(i+1)i^{n-1}$, $(i-1)i^{n-1}$ (together with $i^n$, they belong to the copy of $C_k$ associated with $V_{i^{n-1}}^k$) and the vertices $(i-1)^n$, $(i+1)^n$ (together with $i^n$, they belong to the copy of $C_k$ given by $X$). Therefore, for $n \geq 2$, the vertices of $X$ have the required property. We have to prove that no other vertex has this property.

Notice that a vertex of type $x^tyw$ or $x^tzw$, with $t=1$ and $\{x,y\}\in E$, $\{x,z\}\not\in E$, has less than four distinct neighbours according to Eq. \eqref{regole}, and so it cannot be the unique common vertex of two copies of $C_k$. Hence, for $n=2$, our characterization of $X$ is proved. From now, we assume $n\geq 3$ and we only consider vertices of type $x^tyw$ or $x^tzw$ for $2\leq t\leq n-1$.

We want to prove that a vertex of this type cannot be the unique common vertex of two copies of $C_k$, although it has four distinct adjacent vertices in $\Gamma_{n}^{C_k}$.\\
\indent Let us start by considering words of type $i^s(i\pm 1)w$, with $s\geq 2$ and $w$ possibly empty (i.e., words starting with a block of $i$'s followed by a  neighbour of $i$).\\
\indent Let us focus our attention on the case $i^s(i+ 1)w$ (the case $i^s(i-1)w$ is analogous). Its distinct neighbours are:
\begin{itemize}
\item $(i+1)i^{s-1}(i+1)w$, via the action of $e_i^{-1}$, and $(i-1)i^{s-1}(i+1)w$, via the action of $e_{i-1}$ (together with the vertex $i^s(i+ 1)w$, they both belong to the copy of $C_k$ associated with $V_{i^{s-1}(i+1)w}^k$);
\item $(i+1)^siw$, via the action of $e_i$, and $(i-1)^s(i+1)w$, via the action of $e_{i-1}^{-1}$.
\end{itemize}
We claim that $i^s(i+ 1)w$, $(i+1)^siw$ and $(i-1)^s(i+1)w$ do not belong together to a copy of $C_k$. In particular we claim that we cannot start from $i^s(i+ 1)w$ , pass to $(i+1)^siw$ and go back in $k-1$ steps to $i^s(i+ 1)w$ passing in the last step through $(i-1)^s(i+1)w$ by avoiding vertex repetitions.\\
\indent Notice that by applying $e_i$ to $i^s(i+ 1)w$ we get $(i+1)^siw$. Such a vertex belongs to the copy of $C_k$ corresponding to $V_{(i+1)^{s-1}iw}^k$ (its neighbours in this copy of $C_k$ are obtained by applying the generators $e_{i+1}^{-1}$ and $e_{i}$). The other cycle to which the vertex should belong must be obtained by applying the generator $e_{i+1}$ and $e_{i}^{-1}$. By using the latter we go back to $i^s(i+ 1)w$, by using the former we get $(i+2)^siw$. In the same way by applying $k-2$ times the generators we get a sequence of vertices until we get $(i-1)^siw$. However, this vertex is not equal to the vertex $(i-1)^s(i+1)w$, which was supposed to be the last vertex in our copy of $C_k$. Hence we cannot go back in $k$ steps to $i^s(i+ 1)w$.

\indent Consider now vertices of type $i^szw$, with $s\geq 2$, $z\neq (i\pm 1)$ and $w$ possibly empty (i.e., words starting with a block of $i$'s followed by a letter non adjacent to $i$ in $C_k$). It can be easily seen that such vertices have four distinct neighbours: the vertices $(i\pm 1)i^{s-1}zw$ (together with $i^szw$, they belong to the copy of $C_k$ associated with $V_{i^{s-1}zw}^k$), and the vertices $(i\pm 1)^szw$. For the latter pair of vertices, one can use the same argument as above to show that they cannot live, together with $i^szw$, in a copy of $C_k$.\\
\indent We have thus proved our characterization of $X$ for each $n$, which implies that $X$ is invariant under the action of $\varphi$.

Now suppose that $X$ is fixed (not only invariant) by $\varphi$. It is possible to prove by induction on $n$ that in this case $\varphi$ is the trivial automorphism. The key idea is that the graph $\Gamma_{n}^{C_k}$ is obtained by gluing together, in a suitable way, $k$ copies of the graph $\Gamma_{n-1}^{C_k}$, each consisting of the vertex set $G_i = \{wi: w \in V^{n-1}\}$, with $i=1,\ldots, k$.\\
\indent We have to consider now the case where $X$ is not fixed by $\varphi$. Let us denote by $\varphi_X$ the automorphism of $C_k$ defined by putting $\varphi_X(i) = j$ if $\varphi(i^n)=j^n$. Now, let $\psi$ be the automorphism of $\Gamma_{n}^{C_k}$ induced by $\varphi_X$ as in Proposition \ref{ind}. By definition, the composition of $\varphi$ and $\psi^{-1}$ gives an automorphism of $\Gamma_{n}^{C_k}$ fixing $X$. By the previous discussion, we get $\varphi=\psi$, and the thesis follows.
\end{proof}
From a geometric point of view inherited from $C_k$ via Proposition \ref{ind}, a nontrivial automorphism of $\Gamma_{n}^{C_k}$ is a composition of reflections around the axis of the path connecting vertices of type $i^n$ and $(i+1)^n$, with rotations by $2\pi/k$.
\begin{example}\rm
In Fig. \ref{ciccic} the graphs $\Gamma_2^{C_3}$ (on the left), $\Gamma_3^{C_3}$ (in the middle) and $\Gamma_2^{C_4}$ (on the right) are depicted (notice that the automaton associated with the graph $C_3$ is represented in Case (3) of Example \ref{esempibase}). It can be easily seen that $Aut(\Gamma_2^{C_3})=Aut(\Gamma_2^{C_3})=D_6$ and $Aut(\Gamma_2^{C_4}) = D_8$. Looking, for instance, at the copy $G_2$ of $\Gamma_2^{C_3}$ contained in $\Gamma_3^{C_3}$, we see that when passing from the level $2$ to the level $3$ the edge $\{22,11\}$ (resp. $\{22,33\}$) produces the edges $\{221,112\}$ and $\{222,111\}$ (resp. $\{223,332\}$ and $\{222,333\}$) connecting the copy $G_2$ with the copy $G_1$ (resp. $G_3$); on the other hand, the edges $\{222,112\}$ and $\{222,332\}$ do not appear in $\Gamma_3^{C_3}$.
\begin{center}
\begin{figure}[h]
\begin{center}
\tiny
\psfrag{00}{$11$}\psfrag{11}{$22$}
\psfrag{22}{$33$}\psfrag{33}{$44$}
\psfrag{000}{$111$} \psfrag{111}{$222$}\psfrag{222}{$333$}

\psfrag{002}{$113$}\psfrag{220}{$331$}
\psfrag{110}{$221$}\psfrag{001}{$112$}
\psfrag{221}{$332$} \psfrag{112}{$223$}

\psfrag{a}{$33$}\psfrag{b}{$13$}\psfrag{c}{$31$}\psfrag{d}{$11$}\psfrag{e}{$21$}\psfrag{f}{$12$}\psfrag{g}{$22$}\psfrag{h}{$32$}
\psfrag{i}{$23$}
\psfrag{G1}{$G_1$}\psfrag{G2}{$G_2$}\psfrag{G3}{$G_3$}

\includegraphics[width=0.92\textwidth]{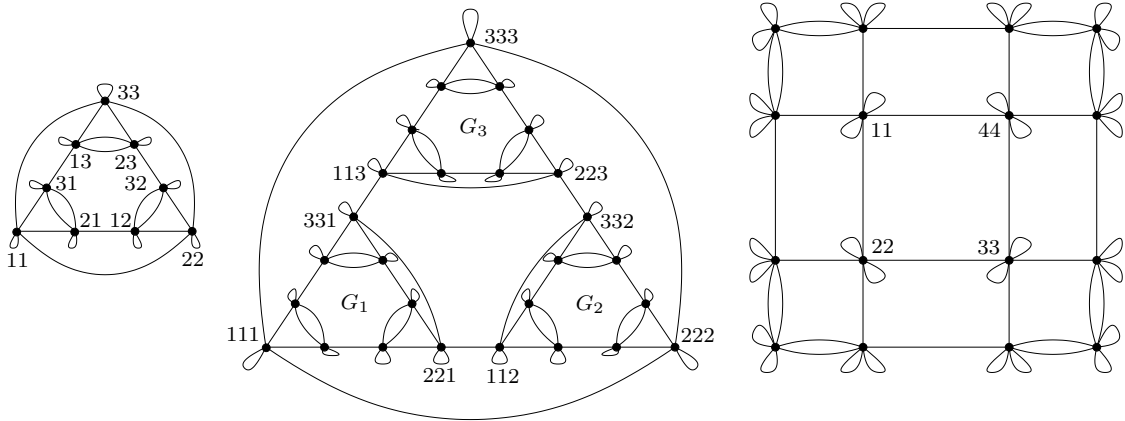}
\end{center}\caption{The graphs $\Gamma_2^{C_3}$, $\Gamma_3^{C_3}$ and $\Gamma_2^{C_4}$.} \label{ciccic}
\end{figure}
\end{center}
\end{example}

\begin{rem}\rm
The previous theorem can be used to show that each isomorphism class of infinite Schreier graphs associated with the action of the group $\mathcal{G}_{C_k}$ contains finitely many graphs. In fact any isomorphism of infinite graphs induces an automorphism of finite graphs  $\Gamma_{n}^{C_k}$ and, since the number of such automorphisms is bounded, one gets the assert. The same phenomenon have been already noticed in \cite{bound} for the Hanoi Towers group.
\end{rem}

\subsection{The case of a path graph $P_k$}\label{paths}
In this subsection, we give a precise description of the diameter and of the automorphism group of the Schreier graph $\Gamma_{n}^{G}$ in the case where $G$ is a path graph.\\
\indent A path $P_k$ on $k$ vertices, that we denote by $\{1,\ldots, k\}$, is a tree with two leaves and $k-2$ vertices of degree $2$ (see Fig. \ref{uu}). We will call extremal the edges containing the two leaves (denoted by $e_1$ and $e_{k-1}$) and internal the other ones. We have already remarked in Example \ref{esempibase} that the group $\mathcal{G}_{P_2}$ is isomorphic to $\mathbb{Z}$ (whose $n$-th Schreier graph is a cyclic graph of length $2^n$) and that the group $\mathcal{G}_{P_3}$ is the so-called Tangled odometer group (whose first four Schreier graphs are drawn in Fig. \ref{primitre} and Fig. \ref{quattro}). \\
\indent Given a graph $G=(V,E)$, we denote by $d(u,v)$ the geodesic distance between $u$ and $v$, that is, the length of a shortest path in $G$ connecting $u$ and $v$. Then the \emph{diameter} of $G$ is defined as
$$
diam (G) = \max\{d(u,v) : u, v \in V\}.
$$
\begin{figure}[h]
\begin{center}
\normalsize
\psfrag{G}{$P_k$}

\footnotesize
\psfrag{0}{$1$}\psfrag{1}{$2$}
\psfrag{2}{$3$}\psfrag{k-2}{$k-1$}
\psfrag{k-1}{$k$}

\psfrag{e1}{$e_1$}\psfrag{e2}{$e_2$}
\psfrag{ek-1}{$e_{k-1}$}
\includegraphics[width=0.7\textwidth]{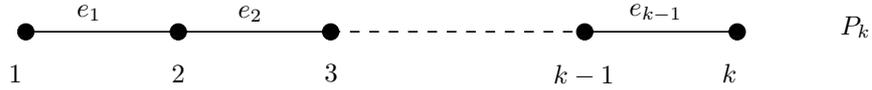}
\end{center}\caption{The path graph $P_k$.} \label{uu}
\end{figure}

First we want to understand the structure of the graphs $\Gamma_{n}^{P_k}$. Fix $k\geq 3$. Since $P_k$ is a tree, all its $k-2$ vertices of degree two are cut-vertices, and so, according to Proposition \ref{cut} they give rise to cut-vertices in $\Gamma_{n}^{P_k}$. The two leaves correspond to loops (see Remark \ref{AA}).
\begin{example}\rm
The Schreier graphs $\Gamma^{P_4}_n$, for $n=1,2,3$ are shown in Fig. \ref{cucu} and Fig. \ref{cucucu}.

\begin{figure}[h]
\begin{center}
\psfrag{0}{$1$}\psfrag{1}{$2$}\psfrag{2}{$3$}\psfrag{3}{$4$}
\tiny \psfrag{11}{$22$}\psfrag{00}{$11$}\psfrag{22}{$33$}\psfrag{33}{$44$}
\psfrag{30}{$41$}\psfrag{03}{$14$}
\includegraphics[width=0.8\textwidth]{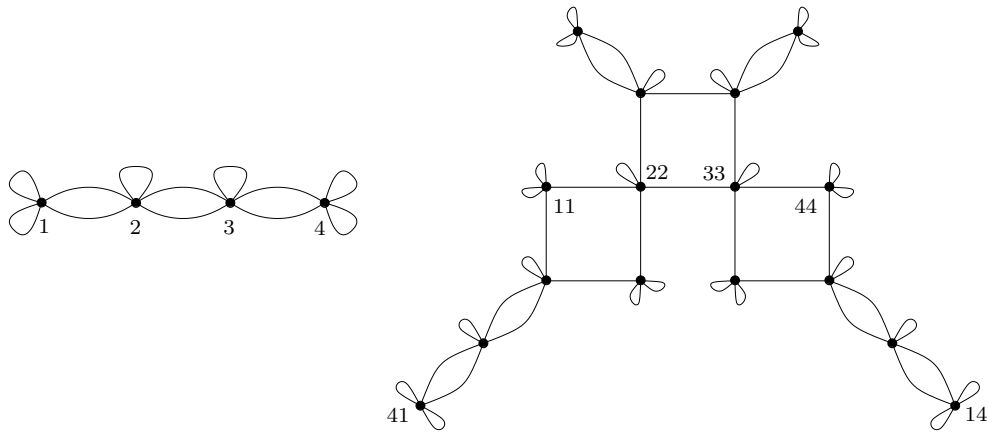}
\end{center}\caption{The Schreier graphs $\Gamma^{P_4}_1$ and $\Gamma^{P_4}_2$.} \label{cucu}
\end{figure}
\newpage

\begin{figure}[h]
\begin{center}
\tiny \psfrag{111}{$222$}\psfrag{000}{$111$}\psfrag{222}{$333$}\psfrag{333}{$444$}  \psfrag{303}{$414$}\psfrag{030}{$141$}
\psfrag{221}{$221$}  \psfrag{334}{$334$}
\includegraphics[width=0.8\textwidth]{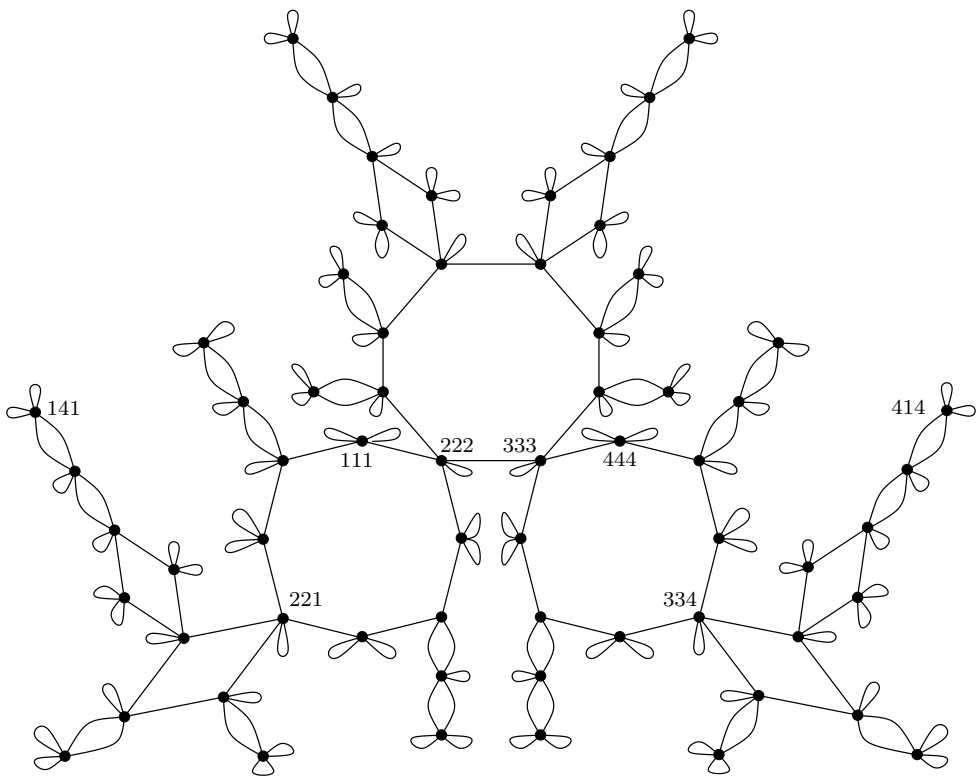}
\end{center}\caption{The Schreier graph $\Gamma^{P_4}_3$.} \label{cucucu}
\end{figure}
Notice that in this case we can observe a \lq\lq wedge\rq\rq shape of the Schreier graphs, contrary to the straight shape of the graphs $\Gamma^{P_3}_n$ shown in Fig. \ref{primitre} and Fig. \ref{quattro}. This property depends on the existence of an internal edge in $P_4$, which does not appear in $P_3$.
\end{example}
Our aim is to describe how one can recursively construct the graph $\Gamma_{n}^{P_k}$ starting from $\Gamma_{n-1}^{P_k}$. Following \cite{PhDBondarenko}, we can proceed as follows.

We take $k$ copies of $\Gamma_{n-1}^{P_k}$ and append to the end of the vertices of the $i$-th copy the letter $i$, for $i=1,\ldots,k$. Then, for each $i=2,\ldots,k-1$, we remove the edges $\{i^n,(i-1)^{n-1}i\}$ and $\{i^n,(i+1)^{n-1}i\}$. We also remove the edges $\{1^{n},2^{n-1}1\}$ and $\{k^n,(k-1)^{n-1}k\}$. Finally, for $i=1,\ldots, k-1$, we join the $i$-th and $(i+1)$-th copies by adding the edges $\{i^{n},(i+ 1)^{n}\}$ and $\{(i+ 1)^{n-1}i,i^{n-1}(i+ 1)\}$. The last operation gives rise to new cycles of doubled length with respect to the level $n-1$.

\begin{example}\rm
In Fig. \ref{figcopie} the construction of $\Gamma_3^{P_4}$ starting from $4$ copies of $\Gamma_2^{P_4}$ is shown. The copies are separated by dotted lines; the deleted edges are represented by dashed lines; the new edges producing cycles of length $8$ are in bold lines.
\begin{figure}[h]
\begin{center}
\tiny \psfrag{111}{$222$}\psfrag{000}{$111$}\psfrag{222}{$333$}\psfrag{333}{$444$}

\psfrag{j}{$332$}\psfrag{k}{$223$}\psfrag{u}{$221$}\psfrag{v}{$112$}  \psfrag{y}{$443$}\psfrag{z}{$334$}
 \psfrag{m}{$141$}\psfrag{n}{$414$}
\scriptsize
\psfrag{c1}{copy $1$}\psfrag{c2}{copy $2$}\psfrag{c3}{copy $3$}\psfrag{c4}{copy $4$}
\includegraphics[width=0.8\textwidth]{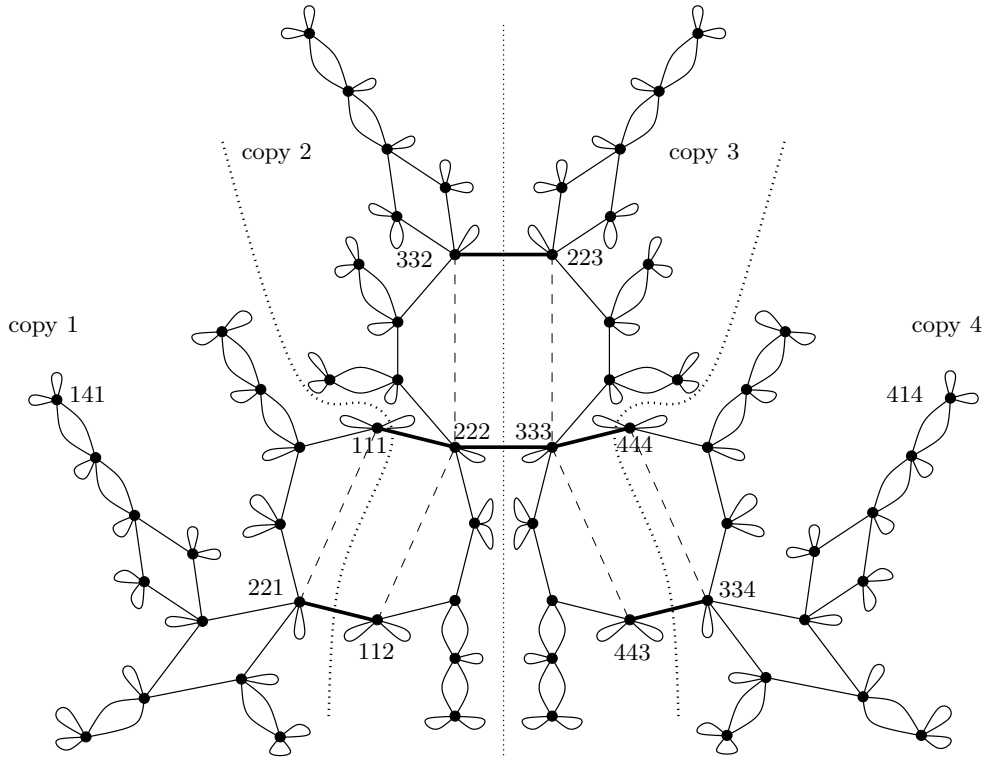}
\end{center}\caption{The construction of $\Gamma^{P_4}_3$ from $\Gamma^{P_4}_2$.} \label{figcopie}
\end{figure}
\end{example}
Recall that each generator $e_i$, $i=1,\ldots, k-1$ corresponds to an edge of $P_k$ connecting two consecutive vertices $i$ and $i+1$ (Fig. \ref{uu}). More effectively, the action of $e_i$ on vertices of types $wxv\in V^n$ with $w\in \{i,i+1\}^k$, $x\neq i,i+1$ and $v$ arbitrary (with $x,v$ possibly empty) gives rise to a cycle of length $2^{k}$, since $e_i$ acts on $\{i,i+1\}^k$ like an Adding machine. For $i=1,\ldots, k-2$, such cycles share vertices of type $(i+1)^kxv$ with cycles corresponding to the action of the generator $e_{i+1}$, which is the only other generator that acts nontrivially on the letter $i+1$. In particular, maximal cycles in $\Gamma_{n}^{P_k}$ are those containing the vertices of type $i^{n}$: for $i=2,\ldots, k-2$, such vertices belong to the two cycles of length $2^n$ and labeled by $e_i$ and $e_{i+1}$, whereas for $i=1$ (resp. $i=k-1$) the vertex $i^n$ belong to the unique maximal cycle labeled by $e_1$ (resp. $e_{k-1}$). It follows that $\Gamma_{n}^{P_k}$ has a cactus structure, where adjacent cycles can be labeled by generators which correspond to incident edges in $P_k$.\\
\indent Let us analyze the behaviour of maximal cycles when constructing $\Gamma_{n}^{P_k}$ from $\Gamma_{n-1}^{P_k}$. Notice that cycles of $\Gamma_{n-1}^{P_k}$ labeled by $e_i$ containing vertices of type $wxv\in V^{n-1}$ with $w\in \{i,i+1\}^k$ and $x\neq i,i+1$ nonempty, also appear in  $\Gamma_{n}^{P_k}$: more precisely, they correspond to cycles containing vertices of type $wxvy\in V^n$ with $w\in \{i,i+1\}^k$, $y=1,\ldots, k$ and $x\neq i,i+1$ nonempty. Maximal cycles in $\Gamma_{n-1}^{P_k}$, which are the cycles containing the vertices $i^{n-1}$, appear in $\Gamma_{n}^{P_k}$. They correspond to nonmaximal cycles containing vertices $i^{n-1}x\in V^n$, for $x\neq i, i+1$ nonempty. For $x=i,i+1$ the generator $e_i$ gives rise to a new bigger (doubled length) unique cycle of length $2^n$ in $\Gamma_{n}^{P_k}$. All these maximal cycles in $\Gamma_{n}^{P_k}$ are connected through the path $1^n, 2^n, \ldots, k^n$ (see, for instance, the path $1^3,2^3,3^3,4^3$ in Fig. \ref{figcopie}).

As explained in Fig. \ref{e1} and Fig. \ref{ei} each new maximal cycle generated by $e_i$ has attached decorations isomorphic to the decorations appended to vertices belonging to the biggest cycle labeled by $e_i$ in $\Gamma_{n-1}^{P_k}$. More precisely, the decorations $D_v$ and $D_w$ corresponding to the vertices  $v=(i+1)^{n-1}i$ and $w=i^{n-1}(i+1)$ of the cycle of length $2^n$ generated by $e_i$ in $\Gamma_{n}^{P_k}$ are isomorphic to the decorations attached to the vertices $(i+1)^{n-1}$ and $i^{n-1}$, respectively, both belonging to the cycle of length $2^{n-1}$ generated by $e_i$ in $\Gamma_{n-1}^{P_k}$.
\begin{figure}[h]
\begin{center}
\normalsize
\psfrag{G}{$\Gamma^{P_k}_{n-1}$}\psfrag{GG}{$\Gamma^{P_k}_n$}

\tiny
\psfrag{e1}{$e_1$}\psfrag{e2}{$e_2$}
\psfrag{a}{$2^{n-2}1$}\psfrag{b}{$2^{n-1}$}
\psfrag{c}{$\ell=2^{n-1}$}\psfrag{d}{$\ell=2^{n-1}$}

\psfrag{e}{$2^{n-1}1$}\psfrag{f}{$2^{n}$}
\psfrag{m}{$2^{n-2}12$}\psfrag{n}{$2^{n-2}11$}\psfrag{h}{$\ell=2^n$}
\includegraphics[width=0.7\textwidth]{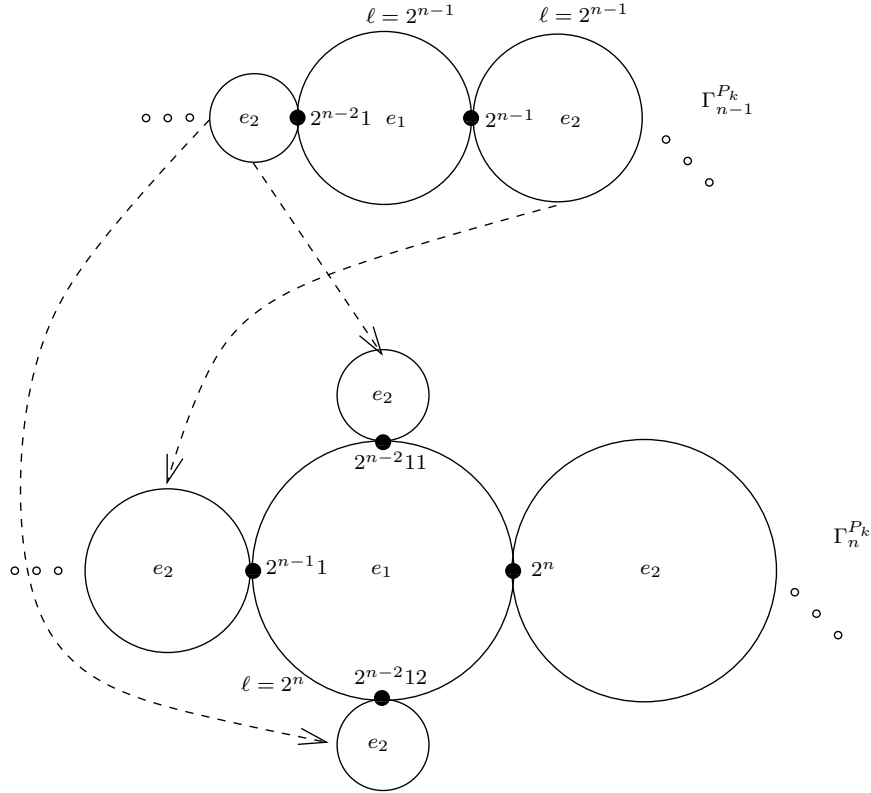}
\end{center}\caption{Change of cycles of extremal edges from level $n-1$ to level $n$.} \label{e1}
\end{figure}
\begin{figure}[h]
\begin{center}
\normalsize
\psfrag{G}{$\Gamma^{P_k}_{n-1}$}\psfrag{GG}{$\Gamma^{P_k}_n$}
  \tiny
\psfrag{e-}{$e_{i-1}$}\psfrag{e}{$e_i$}\psfrag{e+}{$e_{i+1}$}
\psfrag{c}{$\ell=2^{n-1}$}\psfrag{d}{$\ell=2^{n}$}

\psfrag{a}{$i^{n-1}$}\psfrag{b}{$(i+1)^{n-1}$}

\psfrag{n}{$i^{n-1}(i+1)$}\psfrag{m}{$(i+1)^{n-1}i$}
\psfrag{u}{$i^n$}\psfrag{v}{$(i+1)^n$}
\includegraphics[width=0.9\textwidth]{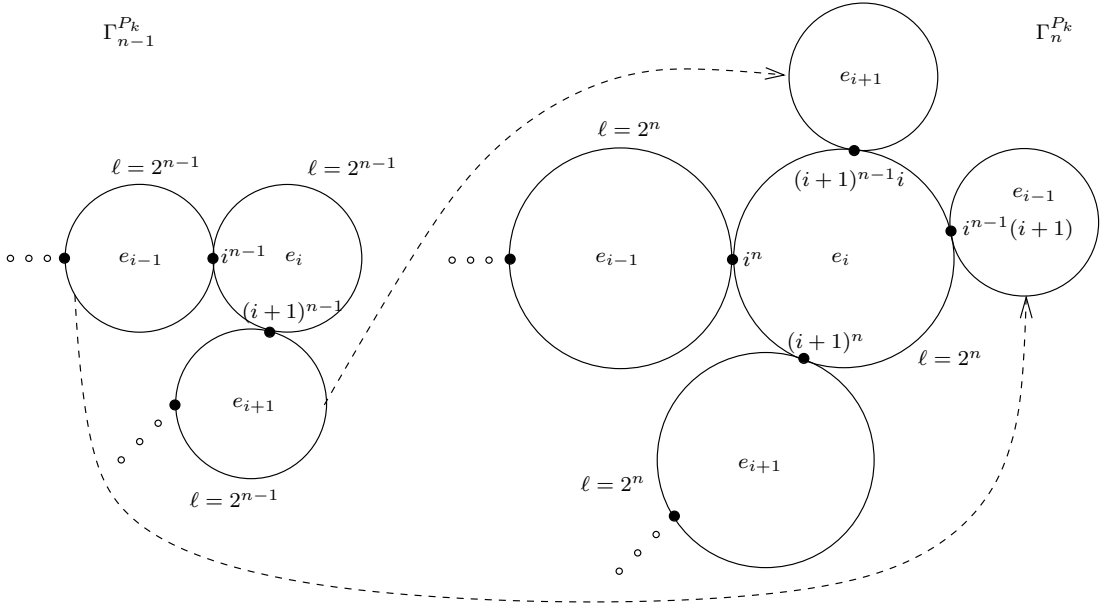}
\end{center}\caption{Change of cycles of internal edges from level $n-1$ to level $n$.} \label{ei}
\end{figure}

It is possible to prove by induction that the diameter of $\Gamma_{n}^{P_k}$ is realized by the distance of the two vertices $a_{k,n}$ and $b_{k,n}$, where
$$
a_{k,n} = \left\{
              \begin{array}{ll}
               (1k)^{(n-1)/2}1  & \hbox{for odd }n \\
                (1k)^{n/2} & \hbox{for even } n
              \end{array}
            \right.  \qquad   b_{k,n} = \left\{
              \begin{array}{ll}
               (k1)^{(n-1)/2}k  & \hbox{for odd } n \\
                (k1)^{n/2} & \hbox{for even } n.
              \end{array}
            \right.
$$
For instance, we have $a_{4,3}=141$ and $b_{4,3}=414$ in Fig. \ref{cucucu}, with $d(a_{4,3},b_{4,3})=19$.\\
\indent In fact, for $n=1$ the statement clearly holds. Suppose now that $n$ is odd (the case $n$ even is similar) and that $a_{k,n-1}=(1k)^{(n-1)/2}$ and $b_{k,n-1}=(k1)^{(n-1)/2}$ realize the diameter in $\Gamma_{n-1}^{P_k}$. Notice that a path from  $a_{k,n-1}$ to  $b_{k,n-1}$ must visit the vertices $2^{n-1}$ and $(k-1)^{n-1}$. In particular $a_{k,n-1}$ is the vertex in $\Gamma_{n-1}^{P_k}$ whose distance from the vertex $2^{n-1}$ is maximal. Analogously $b_{k,n-1}$ is the vertex in $\Gamma_{n-1}^{P_k}$ whose distance from the vertex $(k-1)^{n-1}$ is maximal.

Therefore, passing from $\Gamma_{n-1}^{P_k}$ to $\Gamma_{n}^{P_k}$ we have that $a_{k,n}=a_{k,n-1}1$ is the vertex with maximal distance from $2^{n-1}1$ among the vertices belonging to the copy $1$ of $\Gamma_{n-1}^{P_k}$ within $\Gamma_{n}^{P_k}$; similarly, $b_{k,n}=b_{k,n-1}k$  is the vertex with maximal distance from $(k-1)^{n-1}k$ among the vertices belonging to the copy $k$ of $\Gamma_{n-1}^{P_k}$ within $\Gamma_{n}^{P_k}$.

It follows that the diameter of $\Gamma_{n}^{P_k}$ is realized by a path from $a_{k,n}$ to $b_{k,n}$ given by the sequence of paths
\begin{eqnarray}\label{agostino}
(1k)^{(n-1)/2}1\to_{p_1} 2^{n-1}1\to 2^n\to_{p_\ast} (k-1)^n\to (k-1)^{n-1}k \to_{p_2} (k1)^{(n-1)/2}k.
\end{eqnarray}
Observe that transition $p_\ast$ is the path $2^n, 3^n, \ldots, (k-1)^n$ whose length is $k-3$. Moreover, inside the transitions $p_1$ and $p_2$ there are the transitions
\begin{equation}\label{1}
(k-1)^{n-1}1, \ (k-2)^{n-1}1, \ldots, 3^{n-1}1, \ 2^{n-1}1
\end{equation}
and
\begin{equation}\label{2}
(k-1)^{n-1}k, \ (k-2)^{n-1}k, \ \ldots, 3^{n-1}k, \  2^{n-1}k,
\end{equation}
respectively, which are the analogue at the level $n-1$ of the transition $p_\ast$ described above, so that each of them has length $k-3$: they come from the previous level and are contained in the subgraphs attached at $2^n$ and $(k-1)^n$, respectively. More precisely, the projection of the path $p_1$ on $\Gamma_{n-1}^{P_k}$ consists of a path $p_1'$ from $(1k)^{(n-1)/2}$ to $(k-1)^{n-1}$ followed by the path $p_1''$ given by the projection of \eqref{1}. Similarly, the projection of the path $p_2$ on $\Gamma_{n-1}^{P_k}$ consists of the path $p_2''$  given by the projection of \eqref{2} followed by the path $p_2'$ from $2^{n-1}$ to $(k1)^{(n-1)/2}$.

Denote by $\ell(p)$ the length of the path $p$ and observe that
$$
\ell (p_1') +\ell (p_\ast) +\ell (p_2') = diam (\Gamma_{n-1}^{P_k}).
$$
We are now ready to prove the following theorem.
\begin{theorem}\label{diametro}
Let $P_k$ be the path graph on $k$ vertices, with $k \geq 3$. Then, for every $n\geq 1$:
$$
diam(\Gamma_{n}^{P_k})=2^{n+1}+(k-1)(2n-1)-4n.
$$
\end{theorem}
\begin{proof}
Put $\delta_k(n):=diam(\Gamma_{n}^{P_k})$. As in the preliminary discussion preceding this theorem, we assume that $n$ is odd, the case of an even $n$ being similar. By looking at Eq. \eqref{agostino}, we get $\delta_k(n)=\sum_{i=1}^5 d_i$, where
$$
d_1:=d(a_{k,n}, 2^{n-1}1) = \ell(p_1), \quad  \quad d_2:=d(2^{n-1}1,2^n), \quad  \quad d_3:=d(2^n, (k-1)^n)=\ell(p_\ast)
$$
$$
d_4:=d((k-1)^n,(k-1)^{n-1}k),\qquad  \qquad  d_5:=d((k-1)^{n-1}k, b_{k,n})=\ell(p_2).
$$
From what we said above, we have:
\begin{eqnarray*}
d_1+d_3+d_5 &=& (\ell(p_1')+\ell(p_1'')) + \ell(p_\ast) + (\ell(p_2'') + \ell(p_2'))\\
&=& \delta_k(n-1) + 2(k-3).
\end{eqnarray*}
Since the vertices $2^{n-1}1$ and $2^n$ (resp. $(k-1)^n$ and $(k-1)^{n-1}k$) belong to the maximal cycle generated by $e_1$ (resp. $e_{k-1}$) and they are in opposite positions within this cycle, we have $d_2=d_4=2^{n-1}$. Therefore, we obtain the following recursive description of $\delta_k(n)$:
$$
\left\{
  \begin{array}{ll}
\delta_k(n)=\delta_{k}(n-1)+2(k-3)+2^n\\
\delta_k(1) = k-1.
\end{array}
\right.
$$
A direct computation gives:
\begin{eqnarray*}
\delta_k(n) &=&\delta_k(1) + 2(n-1)(k-3)+\sum_{i=2}^n2^i\\
&=& 2^{n+1}+(k-1)(2n-1)-4n.
\end{eqnarray*}
\end{proof}
In the remaining part of the paper, we give a precise description of the automorphism group of the Schreier graph $\Gamma_{n}^{P_k}$.\\
\indent First of all, notice that a cycle of a given size must be either invariant or moved to another cycle of the same size under the action of an automorphism. Moreover, maximal cycles (of length $2^n$) generated by extremal edges are the only two maximal cycles which are connected to only one cycle of the same size. This implies that they are either invariant or swapped by an automorphism. Let us denote by $A_n$ the set of automorphisms of $\Gamma_{n}^{P_k}$   leaving invariant the two maximal cycles labeled with extremal edges, and by $B_n$ the set of automorphisms of $\Gamma_{n}^{P_k}$  swapping them.

If $\phi\in A_n$, it is possible to prove that actually $\phi$ fixes the vertices in every cycle labeled with internal edges.
The claim easily holds for maximal cycles labeled by internal edges (see for instance the cycle of length $8$ containing the vertices $2^3$ and $3^3$ in Fig. \ref{cucucu}). In fact, they contain a pair of adjacent vertices attached to cycles of the same length and they cannot be swapped; therefore they must be fixed, being invariant cycles with two fixed vertices. As a consequence, $\phi$ acts nontrivially only on the decorations attached to the maximal cycles generated by the action of the internal generators. Then  $\phi$ can be characterized by the automorphisms that it induces on such decorations. Since such decorations already appear in  $\Gamma_{n-1}^{P_k}$ with the same labels and their automorphisms extend to automorphism of $\Gamma_{n-1}^{P_k}$ in $A_{n-1}$, we can conclude by using an inductive argument that all cycles labeled by internal edges in $\Gamma_{n}^{P_k}$ must be fixed by any element  in $A_n$.\\
\indent On the other hand, cycles of length greater or equal to $4$ generated by the action of the extremal generators $e_1$ and $e_{k-1}$ (for instance, the cycles of length $8$ in Fig. \ref{cucucu} containing the vertices $1^3$ and $4^3$) give rise to blocks that have nontrivial symmetries around the vertices $2^tw$, $2^{t-1}1w$ and $(k-1)^tw$, $(k-1)^{t-1}kw$. Such symmetries are generated by the reflection around the axis connecting such vertices (i.e., that keep such vertices fixed). As an example, again in Fig. \ref{cucucu}, one can consider in the leftmost cycle of length $8$ the symmetry around the axis connecting the vertices $222$ and $221$. This implies that each cycle of length greater or equal to $4$ generated by the action of $e_1$ and $e_{k-1}$ gives rise to a nontrivial automorphism of order $2$. Observe that such automorphisms commute, since each one acts nontrivially on a prescribed block and fixes the others. In particular, if $\phi\in A_n$ then it is a composition of these reflections.

Now let us denote by $\psi$ the automorphism of $\Gamma_{n}^{P_k}$ induced, in the sense of Proposition \ref{ind}, by the nontrivial automorphism of $P_k$ switching the vertex $i$ with the vertex $k-i+1$. Clearly $\psi\in B_n$ and $\psi^2=id$. Moreover, for any $\varphi\in B_n$ we have that the composition of $\psi$ with $\varphi$ is in  $A_n$. In particular, every automorphism of $\Gamma_{n}^{P_k}$ is a composition of the aforementioned  reflections in $A_n$ and possibly $\psi$. Moreover, $\psi$ commutes with these reflections.  We are now ready to prove the following result.

\begin{theorem} \label{autogroupfinale}
Let $P_k$ be the path graph on $k$ vertices, with $k\geq 3$. Then, for every $n\geq 2$, the group of automorphisms of $\Gamma_{n}^{P_k}$ is $\mathbb{Z}_2^{\phi_k(n)+1}$, where
$$
\phi_k(n)=\frac{2(k^{n-1}-2k^{n-2}+1)}{k-1}
$$
is the number of cycles of length greater or equal to $4$ in $\Gamma_{n}^{P_k}$ generated by the action of the generators $e_1$ and $e_{k-1}$.
\end{theorem}
\begin{proof}
It follows from the above discussion that we have to count the number of cycles of length greater or equal to $4$ in $\Gamma_{n}^{P_k}$ generated by the action of $e_1$ and $e_{k-1}$. We proceed by induction in order to prove that such a number coincides with $\phi_k(n)$. For $n=2$, we just have the two cycles of length $4$ with vertex sets $\{11,22, 12, 21\}$ and $\{(k-1)^2, k^2, (k-1)k, k(k-1)\}$, with the nontrivial automorphisms flipping $11, 12$ and $k^2, k(k-1)$. Now, passing from $\Gamma_{n-1}^{P_k}$ to $\Gamma_{n}^{P_k}$, all cycles generated by $e_1$ and $e_{k-1}$ are preserved just appending a final letter in $\{1,\ldots, k\}$ to the vertices, except the maximal cycle in $\Gamma_{n-1}^{P_k}$ containing $2^{n-1}$, and the maximal cycle in $\Gamma_{n-1}^{P_k}$ containing $(k-1)^{n-1}$, since each of them will produce a bigger maximal cycle in $\Gamma_{n}^{P_k}$ (the one containing $2^{n-1}1$ and $2^n$, and the one containing  $(k-1)^{n-1}k$ and $(k-1)^n$), so that each of them gives rise to $k-1$ cycles to be taken into account at level $n$. Therefore, we obtain the following recursive formula for $\phi_k(n)$:
$$
\left\{
  \begin{array}{ll}
\phi_k(n)=k\phi_k(n-1)-2\\
\phi_k(2) = 2.
\end{array}
\right.
$$
A direct computation gives:
\begin{eqnarray*}
\phi_k(n)&=&k^{n-2}\phi_k(2)-2\sum_{i=0}^{n-3} k^i\\
&=& \frac{2(k^{n-1}-2k^{n-2}+1)}{k-1}.
\end{eqnarray*}
The claim follows by adding to the $\phi_k(n)$ automorphisms the one induced, in the sense of Proposition \ref{ind}, by the nontrivial automorphism of $P_k$ switching the vertex $i$ with the vertex $k-i+1$.
\end{proof}

\begin{rem}\rm
Theorem \ref{diametro} and Theorem \ref{autogroupfinale} do not work for $k=2$. In this case, we get the Adding machine (see Case (1) in Example \ref{esempibase}) whose Schreier graphs are cycles. More precisely, the $n$-th Schreier graph $\Gamma_{n}^{P_2}$ is the cyclic graph $C_{2^n}$, whose diameter is $2^{n-1}$ and whose automorphism group is isomorphic to the dihedral group $D_{2^{n+1}}$.
\end{rem}

\end{document}